\documentclass[reqno]{amsart}
\usepackage{amsmath,amssymb,amsthm,bm}
\usepackage{tikz,caption,subcaption}
\usepackage{genyoungtabtikz}
\usepackage{ytableau,tikz,varwidth}
\usetikzlibrary{calc}

\usepackage{ytableau}
\usepackage{dsfont}
\def\bbbz{\mathbb Z}
\def\bbbn{\mathbb N}

\def\cch{\mathcal H}
\def\ccsc{\mathcal {SC}}
\def\ccdd{\mathcal {DD}}
\def\ccp{\mathcal P}

\DeclareMathOperator{\BGP}{BG}

\newtheorem{theorem}{Theorem}[section]

\newtheorem{proposition}[theorem]{Proposition}
\newtheorem{definition}[theorem]{Definition}
\newtheorem{corollary}[theorem]{Corollary}
\newtheorem{lemma}[theorem]{Lemma}
\theoremstyle{remark}
\newtheorem{remark}[theorem]{Remark}

\newcommand\bi[2]{{{#1}\atopwithdelims(){#2}}}

\numberwithin{equation}{section}

\begin{document}
\title[ Multiplication theorems for self-conjugate partitions]{Multiplication theorems for self-conjugate partitions}

\author[David Wahiche]{David Wahiche}
\address{Institut Camille Jordan, Universit\'e Claude Bernard Lyon 1,
69622 Villeurbanne Cedex, France}
\email{wahiche@math.univ-lyon1.fr}

\thanks{}

\date{}

\subjclass[2010]{}



\begin{abstract}
In 2011, Han and Ji proved addition-multiplication theorems for integer partitions, from which they derived modular analogues of many classical identities involving hook-length. In the present paper, we prove addition-multiplication theorems for the subset of self-conjugate partitions. Although difficulties arise due to parity questions, we are almost always able to include the BG-rank introduced by Berkovich and Garvan. This gives us as consequences many self-conjugate modular versions of classical hook-lengths identities for partitions. Our tools are mainly based on fine properties of the Littlewood decomposition restricted to self-conjugate partitions.
\end{abstract}

\maketitle

\section{Introduction and notations}
Formulas involving hook-length abound in combinatorics and representation theory. One illustrative example is the hook-length formula discovered in 1954 by Frame, Robinson and Thrall \cite{FRT}, stating the equality between the number $f^\lambda$ of standard Young tableaux of shape $\lambda$ and size $n$, and the number of permutations of $\lbrace 1,\dots,n\rbrace$ divided by the product of the elements of the hook-lengths multiset $\cch(\lambda)$ of $\lambda$, namely:
\begin{equation*}
f^\lambda=\frac{n!}{\displaystyle\prod_{h\in\cch(\lambda)}h}\cdot
\end{equation*}
A much more recent identity is the Nekrasov--Okounkov formula. It was
discovered independently by Nekrasov and Okounkov in their work on random partitions and Seiberg--Witten theory \cite{NO}, and by Westbury \cite{We} in his work on universal characters for $\mathfrak{sl}_n$. This formula is commonly stated as follows:
\begin{equation}\label{NOdebut}
\sum_{\lambda\in\ccp}q^{\lvert \lambda\rvert}\prod_{h\in\cch(\lambda)}\left(1-\frac{z}{h^2}\right)=\prod_{k\geq 1}\left(1-q^k\right)^{z-1},
\end{equation}
where $z$ is a fixed complex number.
This identity was later obtained independently by Han \cite{Ha}, using combinatorial tools and the Macdonald identities for type $A_t$~\cite{Mac}.

Recall that a \textit{partition} $\lambda$ of a positive integer $n$ is a nonincreasing sequence of positive integers $\lambda=(\lambda_1,\lambda_2,\dots,\lambda_\ell)$ such that $\lvert \lambda \rvert := \lambda_1+\lambda_2+\dots+\lambda_\ell = n$. The integers $\lambda_i$ are called the \textit{parts} of $\lambda$, the number of parts $\ell$ being the \textit{length} of $\lambda$, denoted by $\ell(\lambda)$. The well-known generating series for $\ccp$ can also be obtained by \eqref{NOdebut} with $z=0$:
\begin{equation}\label{gspartitions}
\sum_{\lambda\in\ccp}q^{\vert \lambda\vert}=\prod_{j\geq 1}\frac{1}{1-q^j}.
\end{equation}

Each partition can be represented by its Ferrers diagram, which consists in a finite collection of boxes arranged in left-justified rows, with the row lengths in non-increasing order. The \textit{Durfee square} of $\lambda$ is the maximal square fitting in the Ferrers diagram. Its diagonal will be called the main diagonal of $\lambda$. Its size will be denoted $d=d(\lambda):=\max(s | \lambda_s\geq s)$. As an example, in Figure~\ref{fig:ferrers}, the Durfee square of $\lambda=(4,3,3,2)$, which is a partition of $12$ of length $4$, is coloured in red.

\begin{figure}
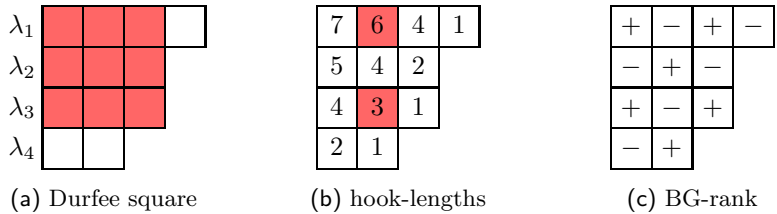

\centering
\begin{subfigure}[t]{.3\textwidth}
\centering
%
\begin{ytableau}
  \none[\lambda_1]  &  *(red!60)  &*(red!60)  & *(red!60) & \\
       \none[\lambda_2] & *(red!60) &*(red!60) &*(red!60) & \none \\
\none[\lambda_3] & *(red!60) &*(red!60)  &*(red!60) &\none \\
\none[\lambda_4] &  &  &\none & \none 
\end{ytableau}
\caption{Durfee square}
\label{fig:ferrers}
\end{subfigure}
\begin{subfigure}[t]{.3\textwidth}
\centering
\begin{ytableau}
  7  &*(red!60) 6 & 4 & 1\\
 5 &4 &2 & \none \\
 4 &*(red!60)3  &1 &\none \\
 2 & 1 &\none & \none \\
\end{ytableau}
\caption{hook-lengths}
\label{fig:hooks}
\end{subfigure}
\begin{subfigure}[t]{.3\textwidth}
\centering
\begin{ytableau}
  +  &-& + & -\\
 - &+ &- & \none \\
 + &- &+ &\none \\
 - & + &\none & \none \\
\end{ytableau}
\caption{BG-rank}
\label{fig:bg}
\end{subfigure}
\caption{Ferrers diagram and some partition statistics}
\label{fig:fig1}
\end{figure}

For each box $v$ in the Ferrers diagram of a partition $\lambda$ (for short we will say for each box $v$ in $\lambda$), one defines the \textit{arm-length} (respectively \textit{leg-length}) as the number of boxes in the same row (respectively in the same column) as $v$ strictly to the right of (respectively strictly below) the box $v$. One defines the \textit{hook-length} of $v$, denoted by $h_v(\lambda)$ or $h_v$, the number of boxes $u$ such that either $u=v$, or $u$ lies strictly below (respectively to the right) of $v$ in the same column (respectively row).
The \textit{hook-length multiset} of $\lambda$, denoted by $\mathcal{H}(\lambda)$, is the multiset of all hook-lengths of $\lambda$. For any positive integer $t$, the multiset of all hook-lengths that are congruent to $0 \pmod t$ is denoted by $\mathcal{H}_t(\lambda)$. Notice that $\mathcal{H}(\lambda)=\mathcal{H}_1(\lambda)$. A partition $\omega$ is a \textit{$t$-core} if $\cch_t(\omega)=\emptyset$. In Figure \ref{fig:hooks}, the hook-lengths of all boxes for the partition $\lambda=(4,3,3,2)$ have been written in their corresponding boxes and the boxes associated with $\mathcal{H}_3(\lambda)$ shaded in red. In the example, we have \linebreak $\mathcal{H}(\lambda)=\lbrace 2,1,4,3,1,5,4,2,7,6,4,1\rbrace$ and $\mathcal{H}_3(\lambda)=\lbrace 3,6\rbrace$.

 A \emph{rim hook} (or border strip, or ribbon) is a connected skew shape containing no $2\times2$ square. The length of a rim hook is the number of boxes in it, and its height is one less than its number of rows. By convention, the height of an empty rim hook is zero.
 
  Recall from the work of Berkovich and Garvan~\cite{BG} that the \emph{BG-rank} of the partition $\lambda$, denoted by BG$(\lambda)$, is defined as follows. First fill each box in the Ferrers diagram of $\lambda$ with alternating $\pm 1$'s along rows and columns beginning with a ``$+1$" in the $(1,1)$ position (see Figure~\ref{fig:bg}). Then sum their values over all the boxes. Note that all boxes belonging to the diagonal of a Ferrers diagram are filled with a ``$+1$". For instance, the BG-rank of $\lambda=(4,3,3,2)$ is $0$.

Let $a$ and $q$ be complex numbers such that $\vert q \vert < 1$. Recall that the $q$-Pochhammer symbol is defined as $(a;q)_0=1$ and for any integer $n\geq 1$:
\begin{flalign*}
&&(a;q)_n&= (1-a)(1-aq) \ldots (1-aq^{n-1}),&&\\
\text{and} && \displaystyle (a;q)_\infty &= \prod_{j\geq 0} (1-aq^j).
\end{flalign*}

A classical bijection in partition theory is the Littlewood decomposition (see for instance \cite[Theorem 2.7.17]{JK}). Roughly speaking, for any positive integer $t$, it transforms $\lambda\in\ccp$ into two components, namely the $t$-core $\omega$ and the $t$-quotient $\underline{\nu}$ (see Section \ref{lit} for precise definitions and properties):
$$
\lambda\in\ccp\mapsto \left(\omega,\underline{\nu}\right)\in\ccp_{(t)}\times \ccp^t.
$$
In \cite{HJ}, Han and Ji underline some important properties of the Littlewood decomposition, which enable them to prove the following multiplication-addition theorem.

\begin{theorem}
\label{Multiplication theorem} {\em \cite[Theorem 1.1]{HJ}}
Let $t$ be a positive integer and set $\rho_1,\rho_2$ two functions defined on $\mathbb{N}$. Let $f_t$ and $g_t$ be the following formal power series:
\begin{align*}
f_t(q)&:=\displaystyle\sum_{\lambda\in\mathcal{P}}q^{|\lambda|}\displaystyle\prod_{h\in\mathcal{H}(\lambda)}\rho_1(th), \nonumber
\\
g_t(q)&:=\displaystyle\sum_{\lambda\in\mathcal{P}}q^{|\lambda|}\displaystyle\prod_{h\in\mathcal{H}(\lambda)}\rho_1(th)\displaystyle\sum_{h\in\mathcal{H}(\lambda)}\rho_2(th). \nonumber \\
\end{align*}
Then we have
$$
\displaystyle\sum_{\lambda\in\mathcal{P}}q^{|\lambda|}x^{|\mathcal{H}_t(\lambda)|}\displaystyle\prod_{h\in\mathcal{H}_t(\lambda)}\rho_1(h)\sum_{h\in\mathcal{H}_t(\lambda)}\rho_2(h)= t\frac{\left(q^t;q^t\right)^t_{\infty}}{\left(q;q\right)_{\infty}}\left(f_t(xq^t)\right)^{t-1}g_t(xq^t).
$$
\end{theorem}

Note that Walsh and Warnaar in \cite{WW} also prove multiplication theorems giving rise to hook-length formulas. They also prove interesting extensions regarding leg-length.

Theorem \ref{Multiplication theorem} gives modular analogues of many classical formulas. For instance, setting $\rho_1(h)=1-z/h^2$ for any complex number $z$ and $\rho_2(h)=1$, it provides the modular analogue of the Nekrasov--Okounkov formula~\eqref{NOdebut} originally proved in~\cite[Theorem 1.2]{Ha}:
\begin{equation}\label{modularNO}
\sum_{\lambda\in\mathcal{P}}q^{\vert \lambda \vert}x^{\vert \cch_t\left(\lambda\right)\vert}\displaystyle\prod_{h\in\mathcal{H}_t(\lambda)}\left(1-\frac{z}{h^2}\right)=\frac{\left(q^t;q^t\right)_\infty ^t}{\left(xq^t;xq^t\right)_\infty ^{t-z/t}\left(q;q\right)_\infty }.
\end{equation}


In the present work, we extend Theorem~\ref{Multiplication theorem} to an important subset of $\ccp$, namely the self-conjugate partitions, and derive several applications regarding these. Recall that the \textit{conjugate} of $\lambda$, denoted $\lambda'$, is defined by its parts $\lambda_i' = \#\{j, \lambda_j \geq i\}$ for \linebreak$1\leq i \leq \ell(\lambda)$. For instance in Figure~\ref{fig:fig1}, the conjugate of $\lambda=(4,3,3,2)$ is \linebreak$\lambda'=(4,4,3,1)$. A partition $\lambda$ is said to be \textit{self-conjugate} if it satisfies $\lambda=\lambda'$. 

We denote the set of self-conjugate partitions by $\mathcal{SC}$. This subset of partitions has been of particular interest within the works of P\'etr\'eolle \cite{MP,Pe} where two Nekrasov--Okounkov type formulas for $\tilde{C}$ and $\tilde{C}$\v{ } are derived. See also the work of Han--Xiong \cite{HX} or Cho--Huh--Sohn \cite{CHS}. The already mentioned Littlewood decomposition, when restricted to $\ccsc$, also has interesting properties and can be stated as follows (see for instance \cite{GKS,MP}):
$$\begin{array}{rcll}
\lambda\in\ccsc &\mapsto &\left(\omega,\underline{\tilde{\nu}}\right)\in\ccsc_{(t)}\times\ccp^{t/2}\quad&\text{if $t$ even,}\\
\lambda\in\ccsc &\mapsto &\left(\omega,\underline{\tilde{\nu}},\mu\right)\in\ccsc_{(t)}\times\ccp^{(t-1)/2}\times\ccsc\quad&\text{if $t$ odd.}\end{array}$$
Indeed, as will be detailed in Section~\ref{lit}, in the particular case of self-conjugate partitions, elements of the $t$-quotient $\underline{\nu}\in\ccp^t$ can be gathered two by two through conjugation (except $\nu^{((t-1)/2)}$ when $t$ is odd), therefore yielding the above vectors $\underline{\tilde{\nu}}$ and $(\underline{\tilde{\nu}},\mu)$.

As can be seen above, to provide an analogue of Theorem~\ref{Multiplication theorem} for self-conjugate partitions, the $t$ even case is simpler to handle, therefore we first restrict ourselves to this setting. Nevertheless, it yields a slightly more general result than Theorem~\ref{Multiplication theorem}, as the BG-rank can be incorporated. 
\begin{theorem}\label{bg4multiSC}
Let $t$ be a positive even integer and set $\rho_1,\rho_2$ two functions defined on $\mathbb{N}$. Let $f_{t}$ and $g_{t}$ be the formal power series defined as:
\begin{align*}
f_{t}(q)&:=\sum_{\nu\in\mathcal{P}}q^{|\nu|}\prod_{h\in\mathcal{H}(\nu)}\rho_1(th)^2,\\
g_{t}(q)&:=\sum_{\nu\in\mathcal{P}}q^{|\nu|}\prod_{h\in\mathcal{H}(\nu)}\rho_1(th)^2\sum_{h\in\mathcal{H}(\nu)}\rho_2(th).
\end{align*}
Then we have
\begin{multline*}
\sum_{\lambda\in \ccsc}q^{|\lambda|}x^{|\mathcal{H}_{t}(\lambda)|}b^{\BGP(\lambda)}\prod_{h\in\mathcal{H}_{t}(\lambda)}\rho_1(h)\sum_{h\in\mathcal{H}_{t}(\lambda)}\rho_2(h)\\
=t\left(f_{t}(x^2q^{2t})\right)^{t/2-1}g_{t}(x^2q^{2t})
\left(q^{2t};q^{2t}\right)_{\infty}^{t/2}\left(-bq;q^4\right)_{\infty}\left(-q^3/b;q^4\right)_{\infty}.
\end{multline*}
\end{theorem}

\begin{remark}  Note that the functions $f_t$ and $g_t$ in Theorem \ref{bg4multiSC} are close to the ones in Theorem \ref{Multiplication theorem}, the explanation is that when $t$ is even, there is no additional self-conjugate partition $\mu$ in the Littlewood decomposition.
\end{remark}
We will derive several consequences of this result, including a new trivariate generating function for $\ccsc$, new hook-length formulas, new modular versions of the Han--Carde--Loubert--Potechin--Sanborn, the Nekrasov--Okounkov, the Bessenrodt--Bacher--Manivel, the Okada--Panova, and the Stanley--Panova formulas. Among them, we highlight here the self-conjugate version of~\eqref{modularNO}.
\begin{corollary}\label{NOsign}
For any complex number $z$ and $t$ an even positive integer, we have:
\begin{multline*}
\sum_{\lambda\in\ccsc}q^{|\lambda|}x^{|\mathcal{H}_t(\lambda)|}b^{\BGP(\lambda)}\prod_{h\in\cch_t(\lambda)}\left(1-\frac{z}{h^2}\right)^{1/2}\\=\left(x^{2}q^{2t};x^{2}q^{2t}\right)_\infty^{(z/t-t)/2}\left(q^{2t};q^{2t}\right)_{\infty}^{t/2}\left(-bq;q^4\right)_{\infty}\left(-q^3/b;q^4\right)_{\infty}.
\end{multline*}
\end{corollary}
As some combinatorial signs naturally appear in the work of P\'etr\'eolle regarding Nekrasov--Okounkov type formulas for self-conjugate partitions, we will also prove a signed refinement of Theorem~\ref{bg4multiSC} (see Theorem~\ref{signmultisc} in Section~\ref{refine}, which actually generalizes Theorem~\ref{bg4multiSC}).

It is also possible to prove a result similar to Theorem~\ref{bg4multiSC} when $t$ is odd; nevertheless more difficulties arise due to the additional $\mu\in\ccsc$ appearing in the Littlewood decomposition. However, as will be seen later, the subset of $\ccsc$ for which $\mu$ is empty, can be handled almost similarly as for Theorem~\ref{bg4multiSC} (see Theorem~\ref{thm:bgt} in Section~\ref{sec:odd}). The interesting thing here is that this subset of $\ccsc$ actually corresponds to partitions called $\BGP_t$ in~\cite{Bern}, which are algebraically involved in representation theory of the symmetric group over a field of characteristic $t$ when $t$ is an odd prime number. \\

This paper is organized as follows. In Section~\ref{lit}, we provide the necessary background and properties regarding the Littlewood decomposition for self-conjugate partitions. Section~\ref{sec:mult} is devoted to the proof of Theorem~\ref{bg4multiSC}, together with some useful special cases. Many interesting modular self-conjugate analogues of the above mentioned classical formulas are then listed and proved in Section~\ref{applications}. In Section~\ref{refine}, our signed generalization of Theorem~\ref{bg4multiSC} is proved, and finally in Section~\ref{sec:odd} we study the odd case.

\section{Combinatorial properties of the Littlewood decomposition on self-conjugate partitions} \label{lit}

In this section, we use the formalism of Han and Ji in \cite{HJ}. Recall that a partition $\mu$ is a $t$-core if it has no hook that is a multiple of $t$. For any $A\subset\ccp$, we denote by $A_{(t)}$ the subset of elements of $A$ that are $t$-cores. For example, the only $2$-cores are the ``staircase" partitions $(k,k-1,\dots,1)$, for any positive integer $k$, which are also the only $\ccsc$ $2$-cores.

Let $\partial \lambda$ be the border of the Ferrers diagram of $\lambda$. Each step on $\partial\lambda$ is either horizontal or vertical. Encode the walk along the border from the South-West to the North-East as depicted in Figure \ref{fig:word}: take ``$0$" for a vertical step and ``$1$" for a horizontal step. This yields a $0/1$ sequence denoted \textbf{$s(\lambda)$}. The resulting word $s(\lambda)$ over the $\lbrace 0,1\rbrace$ alphabet:
\begin{itemize}
\item  contains infinitely many ``$0$"'s (respectively ``$1$"'s) at the beginning (respectively the end),
\item is indexed by $\bbbz$,
\item  and is written $(c_i)_{i\in\mathbb{Z}}$.  
\end{itemize}
This writing as a sequence is not unique since for any $k$, sequences $(c_{k+i})_{i\in\mathbb{Z}}$ encode the same partition. Hence it is necessary for that encoding to be bijective to set the index $0$ uniquely. To tackle that issue, we set the index $0$ when the number of ``$0$"'s on and to the right of that index is equal to the number of ``$1$"'s to the left. In other words, the number of horizontal steps along $\partial\lambda$ corresponding to a ``$1$" of negative index in $(c_i)_{i\in\bbbz}$ must be equal to the number of vertical steps corresponding to ``$0$"'s of nonnegative index in $(c_i)_{i\in\bbbz}$ along $\partial\lambda$. The delimitation between the letter of index $-1$ and the one of index $0$ is called the \textit{median} of the word, marked by a $\mid$ symbol. The size of the Durfee square is then equal to the number of ``$1$"'s of negative index. Hence a partition is bijectively associated by the application $s$ to the word:
\begin{align*}
s(\lambda)=(c_i)_{i\in\mathbb{Z}}=\left(\ldots c_{-2}c_{-1}|c_0c_1c_2\ldots\right), \intertext{where $c_i\in\lbrace 0,1\rbrace$ for any $i\in\bbbz$, and such that}
\#\{i\leq-1,c_i=1\}= \#\{i\geq0,c_i=0\}.
\end{align*}

Moreover, this application maps bijectively a box $u$ of hook-length $h_u$ of the Ferrers diagram of $\lambda$ to a pair of indices $(i_u,j_u)\in\bbbz^2$ of the word $s(\lambda)$ such that
\begin{itemize}
\item  $i_u<j_u$,
\item $c_{i_u}=1$, $c_{j_u}=0$
\item $j_u-i_u=h_u$.
\end{itemize}  
The following lemma will be useful in Section \ref{refine}.
\begin{lemma}\label{lemdurf}
Set $\lambda\in\ccp$ and $s(\lambda)$ its corresponding word. Let $u$ be a box of the Ferrers diagram of $\lambda$. Let $(i_u,j_u)\in\bbbz^2$ be the indices in $s(\lambda)$ associated with $u$. Then $u$ is a box strictly above the main diagonal in the Ferrers diagram of $\lambda$ if and only if $|i_u|\leq |j_u|$.
\end{lemma}

\begin{proof}
Let $u$ be a box and $(i,j)\in\bbbz^2$ the corresponding indices in $s(\lambda)=(c_k)_{k\in\bbbz}$ such that $c_{i_u}=1$ and $c_{j_u}=0$. Assume that $i_u$ and $j_u$ have the same sign. This is equivalent to the fact that the hook defined by the sequence $c_{i_u}\dots c_{j_u}$ begins and ends on the same side of the median of $s(\lambda)$.

Then the box $u$ associated with this hook is either below the Durfee square or to its right. Hence $u$ is below when $i_u$ and $j_u$ are negative as we also know that $i_u<j_u$, then $|j_u|<|i_u|$. If $u$ is to the right of the Durfee square, which is above the main diagonal of the Ferrers diagram, then both $i_u$ and $j_u$ are nonnegative. This implies that $|j_u|>|i_u|$.

Now, if we consider the case $i_u<0\leq j_u$, the box $u$ is in the Durfee square. The sequences $c_{i_u}\dots c_{-1}$ of length $|i_u|$ and $c_0\dots c_{j_u}$ of length $j_u+1$ correspond to the number of steps before, respectively after, the corner of the Durfee square. Moreover $u$ is below the main diagonal if and only if the number of steps before the Durfee square is greater or equal to the number of steps after. Hence it is equivalent to $|i_u|\geq |j_u|+1$.
\end{proof}
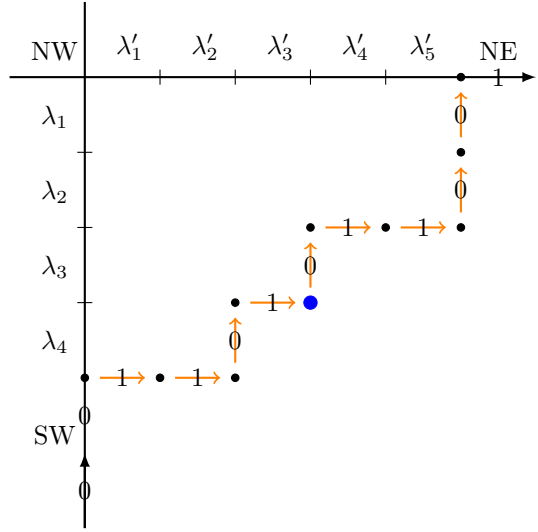
\begin{figure}[h!]
\centering
\begin{tikzpicture}
    [
        dot/.style={circle,draw=black, fill,inner sep=1pt},
    ]

\foreach \x in {0,...,2}{
    \node[dot] at (\x,-4){ };
}

\foreach \x in {.2,1.2}
    \draw[->,thick,orange] (\x,-4) -- (\x+.6,-4);
\foreach \x in {3.2,4.2}
    \draw[->,thick,orange] (\x,-2) -- (\x+.6,-2);
\foreach \y in {1.8,.8}
    \draw[->,thick,orange] (5,-\y) -- (5,-\y+.6);
\draw[->,thick,orange] (2,-3.8) -- (2,-3.8+.6);
\draw[->,thick,orange] (3,-2.8) -- (3,-2.8+.6);
\draw[->,thick,orange] (2.2,-3) -- (2.2+.6,-3);

\node[dot] at (2,-3){};
\foreach \x in {3,...,5}
    \node[dot] at (\x,-2){};
\foreach \x in {1,...,4}
    \draw (\x,-.1) -- node[above,xshift=-0.4cm,yshift=1mm] {$\lambda_\x'$} (\x,+.1);

\node[above,xshift=0.5cm,yshift=1mm] at (4,0) {$\lambda_5'$};
\node[above,xshift=0.5cm,yshift=1mm] at (5,0) {NE};
\node[above,xshift=-4mm,yshift=1mm] at (0,0) {NW};

\foreach \y in {1,...,3}
    \draw (.1,-\y) -- node[above,xshift=-4mm,yshift=0.2cm] {$\lambda_\y$} (-.1,-\y);
\node[above,xshift=-4mm,yshift=2mm] at (0,-4) {$\lambda_4$};
\node[above,xshift=-4mm] at (0,-5) {SW};
\node at (0,-4.5) {0};  
\node at (0,-5.5) {0};
\node at (2,-3.5) {0};
\node at (3,-2.5) {0};
\node at (5,-1.5) {0};
\node at (5,-0.5) {0};

\node at (0.5,-4) {1};
\node at (1.5,-4) {1};   
\node at (4.5,-2) {1};
\node at (3.5,-2) {1};   
\node at (5.5,-0) {1};
\node at (2.5,-3) {1};

\node[dot] at (5,-1){};  
\node[dot] at (5,0){};
\draw[->,thick,-latex] (0,-6) -- (0,-5);
\draw[thick] (0,-6) -- (0,1);
\draw[->,thick,-latex] (-1,0) -- (6,0);
\node[circle,draw=blue,fill=blue,inner sep=0pt,minimum size=5pt] at (3,-3){};

\end{tikzpicture}
\caption{$\partial\lambda$ and its binary correspondence for $\lambda=(5,5,3,2)$.}
\label{fig:word}
\end{figure}
Now we recall the following classical map, often called the Littlewood decomposition (see for instance \cite{GKS,HJ}).

\begin{definition}\label{defphi}{\em
 Let $t \geq 2$ be an integer and consider:\\
$$\begin{array}{l|rcl}
\Phi_t: & \mathcal{P} & \to & \mathcal{P}_{(t)} \times \mathcal{P}^t \\
& \lambda & \mapsto & (\omega,\nu^{(0)},\ldots,\nu^{(t-1)}),
\end{array}$$
where if we set $s(\lambda)=\left(c_i\right)_{i\in\bbbz}$, then for all $k\in\lbrace 0,\dots,t-1\rbrace$, one has \linebreak $\nu^{(k)}:=s^{-1}\left(\left(c_{ti+k}\right)_{i\in\bbbz}\right)$. The tuple $\underline{\nu}=\left(\nu^{(0)},\ldots,\nu^{(t-1)}\right)$ is called the $t$-quotient of $\lambda$ and is denoted by \textit{$quot_t(\lambda)$}, while $\omega$ is the $t$-core of $\lambda$ denoted by \textit{$core_t(\lambda)$}.}
\end{definition}
Obtaining the $t$-quotient is straightforward from $s(\lambda)=\left(c_i\right)_{i\in\bbbz}$: we just look at subwords with indices congruent to the same values modulo $t$. The sequence $10$ within these subwords are replaced iteratively by $01$ until the subwords are all the infinite sequence of ``$0$"'s before the infinite sequence of ``$1$"'s (in fact it consists in removing all rim hooks in $\lambda$ of length congruent to $0\pmod t$). Then $\omega$ is the partition corresponding to the word which has the subwords$\pmod t$ obtained after the removal of the $10$ sequences.
For example, if we take $\lambda = (4,4,3,2) \text{ and } t=3$, then $s(\lambda)=\ldots \color{red}{0} \color{blue}{0} \color{green}{1} \color{red}{1} \color{blue}{0}\color{green}{1} \color{black}| \color{red}{0} \color{blue}{1} \color{green}{0} \color{red}{0} \color{blue}{1} \color{green}{1}\color{black}\ldots$
\begin{align*}
\begin{array}{rc|rcl}
s\left(\nu^{(0)}\right)=\ldots \color{red} 001 \color{black}| \color{red}001\color{black}\ldots& &s\left(w_{0}\right)=\ldots \color{red} 000 \color{black}| \color{red}011\color{black}\ldots,\\
 s\left(\nu^{(1)}\right)=\ldots \color{blue} 000 \color{black}| \color{blue}111\color{black}\ldots& \longmapsto& s\left(w_{1}\right)=\ldots \color{blue} 000 \color{black}| \color{blue}111\color{black}\ldots , \\
 s\left(\nu^{(2)}\right)=\ldots \color{green} 011 \color{black}| \color{green}011\color{black}\ldots & & s\left(w_{2}\right)=\ldots \color{green} 001 \color{black}| \color{green}111\color{black}\ldots .
\end{array}\\
\end{align*}
Thus
\begin{center}
$s(\omega)=\ldots \color{red}{0} \color{blue}{0} \color{green}{0} \color{red}{0} \color{blue}{0}\color{green}{1} \color{black}| \color{red}{0} \color{blue}{1} \color{green}{1} \color{red}{1} \color{blue}{1} \color{green}{1}\color{black}\ldots $
\end{center}
and
$$
quot_3(\lambda)=\left(\nu^{(0)},\nu^{(1)},\nu^{(2)}\right)=\left((1,1),\emptyset,(2)\right),\ core_3(\lambda)= \omega=(1)
$$

The following properties of the Littlewood decomposition are given in \cite{HJ}.

\begin{proposition}\cite[Theorem 2.1]{HJ}
\label{Littlewood} Let $t$ be a positive integer. The Littlewood decomposition $\Phi_t$ maps bijectively a partition $\lambda$ to $\left(\omega,\nu^{(0)},\dots,\nu^{(t-1)}\right)$ such that:
\begin{align*}
&(P1)\quad \omega \text{ is a $t$-core and }\nu^{(0)},\dots,\nu^{(t-1)} \text{are partitions},\\
&(P2) \quad |\lambda|=|\omega|+t\sum_{i=0}^{t-1} |\nu^{(i)}|,\\
&(P3)\quad \mathcal{H}_t(\lambda)=t\mathcal{H}(\underline{\nu}),
\intertext{where, for a multiset $S$,}
&\quad tS:=\{ts,s\in S\}\quad\text{and}\quad \mathcal{H}(\underline{\nu}):=\bigcup\limits_{i=0}^{t	-1}\cch(\nu^{(i)}).\notag
\end{align*}
\end{proposition}
and the first part of their Theorem 2.2 which reads as follows:
\begin{proposition}\cite[Theorem 2.2]{HJ}\label{bghj}
When $t=2$, the Littlewood decomposition $\Phi_2$ has the further two properties:
\begin{align*}
&(P4) \quad\BGP(\lambda)=\begin{cases}
\frac{\ell(\omega)+1}{2}\quad \text{if}\quad \BGP(\lambda)>0,\\
-\frac{\ell(\omega)}{2} \quad \text{if}\quad \BGP(\lambda)\leq 0.
\end{cases}
\end{align*}
\end{proposition}

Now we discuss the Littlewood decomposition for $\ccsc$ partitions. Let $t$ be a positive integer, take $\lambda\in \ccsc$, and set $s(\lambda)=(c_i)_{i \in\mathbb{Z}}\in \lbrace 0,1\rbrace^\bbbz$ and \linebreak $(\omega,\underline{\nu})=\left(core_t(\lambda),quot_t(\lambda)\right)$. Then we have (see for instance \cite{GKS,Pe}):
\begin{eqnarray}
\lambda\in \ccsc &\iff & \forall i_0 \in \lbrace 0,\ldots,t-1\rbrace,\forall j \in \mathbb{N},c_{i_0+jt}=1-c_{-i_0-jt-1}\notag\\
&\iff & \forall i_0 \in \lbrace 0,\ldots,t-1\rbrace,\forall j \in \mathbb{N},c_{i_0+jt}=1-c_{t-(i_0+1)-t(j-1)}\label{mot}\\
&\iff &\forall i_0 \in \left\lbrace 0,\ldots,t-1 \right\rbrace, \nu^{(i_0)}=\left(\nu^{(t-i_0-1)}\right)'\quad \text{and}\quad \omega \in \ccsc_{(t)}\notag .\end{eqnarray}

Therefore $\lambda$ is uniquely defined if its $t$-core is known as well as the $\left\lfloor t/2\right\rfloor$ first elements of its quotient, which are partitions without any constraints. It implies that if $t$ is even, there is a one-to-one correspondence between a self-conjugate partition and a pair made of one $\ccsc$ $t-$core and $t/2$ generic partitions. If $t$ is odd, the Littlewood decomposition is a one to one correspondence between a self-conjugate partition and a triple made of one $\ccsc$ $t-$core, $(t-1)/2$ generic partitions and a self-conjugate partition $\mu=\nu^{((t-1)/2)}$.
Hence the analogues of the above theorems when applied to self-conjugate partitions are as follows.

\begin{proposition}\cite[Lemma 4.7]{MP}\label{SCLittlewood}
Let $t$ be a positive integer. The Littlewood decomposition $\Phi_t$ maps a self-conjugate partition $\lambda$ to $\left(\omega,\nu^{(0)},\dots,\nu^{(t-1)}\right)=(\omega,\underline{\nu})$ such that:
\begin{align*}
&(SC1)\quad \text{the first component } \omega \text{ is a $\ccsc$ $t$-core and }\nu^{(0)},\dots,\nu^{(t-1)} \text{are partitions},\\
&(SC2) \quad \forall j \in \left\lbrace 0,\dots,\left\lfloor t/2 \right\rfloor-1\right\rbrace, \nu^{(j)}=\left(\nu^{(t-1-j)}\right)',\\
&(SC'2)\quad  \text{if t is odd, } \nu^{\left((t-1)/2\right)}=\left(\nu^{\left((t-1)/2\right)}\right)'=:\mu, \\
&(SC3) \quad |\lambda|=\begin{cases}\displaystyle|\omega|+2t\sum_{i=0}^{t/2-1} \lvert\nu^{(i)}\rvert \quad
 \text{if t is even},\\
\displaystyle|\omega|+2t\sum_{i=0}^{(t-1)/2-1} \lvert\nu^{(i)}\rvert+t\lvert \mu\rvert \quad \text{if t is odd},
\end{cases}\\
&(SC4)\quad \mathcal{H}_t(\lambda)=t\mathcal{H}(\underline{\nu}).
\end{align*}
\end{proposition}

The set $D(\lambda)=\lbrace h_{(i,i)}(\lambda),i=1,2,\dots\rbrace$ is called the \textit{set of main diagonal hook-lengths of $\lambda$}. For short, we will denote $h_{(i,i)}$ by $\delta_i$. It is clear that if $\lambda\in\ccsc$, then $D(\lambda)$ determines $\lambda$, and elements of $D(\lambda)$ are all distinct and odd. Hence, as observed in \cite{CHS}, for a self-conjugate partition $\lambda$, the set $D(\lambda)$ can be divided into the following two disjoint subsets:
\begin{align*}
D_1(\lambda)&:=\lbrace \delta_i \in D(\lambda): \delta_i \equiv 1 \pmod 4 \rbrace, \\ 
D_2(\lambda)&:=\lbrace \delta_i \in D(\lambda): \delta_i \equiv 3 \pmod 4 \rbrace .
\end{align*}

We have the following result.
\begin{lemma}\label{BGSC}
For a self-conjugate partition $\lambda$, set $r:=\vert D_1(\lambda)\vert$ and $s:=\vert D_3(\lambda)\vert$. Then
\[\BGP(\lambda)=r-s.\]
\end{lemma}
\begin{proof}
Set $a_1>a_2>\dots>a_r \geq 0$ and $b_1>b_2>\dots>b_s\geq 0$ integers such that:
\begin{align*}
D_1(\lambda)&=\lbrace 4a_1+1,\dots,4a_r+1 \rbrace, \\ 
D_2(\lambda)&=\lbrace 4b_1+3,\dots,4b_s+3 \rbrace .
\end{align*}
Let us consider a hook in the main diagonal of $\lambda$ whose length is $4a+1$ for a nonnegative integer $a$. Then its leg and arm are both of length $2a$. As the BG-rank alternates in sign, we have $\BGP(4a+1)=1$. In the same way, we can observe that $\BGP(4b+3)=-1$ for any main diagonal hook-length $4b+3\in D_2(\lambda)$.
Hence 
\[\BGP(\lambda)=\sum_{i=1}^r\BGP(4a_i+1)+\sum_{j=1}^s \BGP(4b_j+3)=r-s.\]\end{proof}

\begin{remark}
Note that as its diagonal is filled with ``$+1$", we can consider $\lambda$ hook by hook. In the following example are depicted two hooks of length congruent to $1\pmod 4$ and $3\pmod 4$ respectively.
\begin{figure}[h!]
\begin{subfigure}[b]{.45\textwidth}
\centering
\young(+-+-,-,+,-)
\caption{A hook of length $7=4\times 1+3$.}
\end{subfigure}
\begin{subfigure}[b]{.45\textwidth}
\centering
\young(+-+-+,-,+,-,+)
\caption{A hook of length $9=4\times 2 +1$.}
\end{subfigure}
\end{figure}
\end{remark}

In the case $t=2$, we can combine Lemma \ref{BGSC} and Proposition \ref{bghj} $(P4)$ to derive the following additional result.
\begin{proposition}\label{sc4}
The Littlewood decomposition $\Phi_2$ has the further property:
\begin{align*}
&(SC5)\quad \BGP(\lambda)=r-s=\begin{cases}
\frac{\ell(\omega)+1}{2}\quad \text{if}\quad \BGP(\lambda)>0,\\
-\frac{\ell(\omega)}{2} \quad \text{if}\quad \BGP(\lambda)\leq 0.
\end{cases}
\end{align*}
\end{proposition}
%

\section{Multiplication-addition theorems for self-conjugate partitions}\label{sec:mult}

In this section, we prove Theorem \ref{bg4multiSC} stated in the introduction and we exhibit some interesting special cases.
\subsection{Proof of Theorem \ref{bg4multiSC}}

\noindent Let $t$ be a fixed positive even integer. Let $\rho_1$ and $\rho_2$ be two functions defined on $\bbbn$. First we will compute the term
\begin{equation}
\displaystyle\sum_{\substack{\lambda\in \ccsc\\core_t(\lambda)=\omega}}q^{|\lambda|}x^{|\mathcal{H}_{t}(\lambda)|}b^{\BGP(\lambda)}\prod_{h\in\mathcal{H}_{t}(\lambda)}\rho_1(h)\sum_{h\in\mathcal{H}_{t}(\lambda)}\rho_2(h),\label{term}
\end{equation}
where $\omega\in\ccsc_{(t)}$ is fixed. Let us remark that for $\lambda\in \ccsc$ and $\omega=core_{t}(\lambda)$, one has $\BGP(\lambda)=\BGP(\omega)$. Indeed $\omega$ is obtained by removing from $\lambda$ ribbons of even length $t$ and these have BG-rank $0$. Hence \eqref{term} can be rewritten as follows
$$b^{\BGP(\omega)}q^{|\omega|}\displaystyle\sum_{\substack{\lambda\in \ccsc\\core_t(\lambda)=\omega}}q^{|\lambda|-|\omega|}x^{|\mathcal{H}_{t}(\lambda)|}\prod_{h\in\mathcal{H}_{t}(\lambda)}\rho_1(h)\sum_{h\in\mathcal{H}_{t}(\lambda)}\rho_2(h).$$

Hence using properties $(SC3)$ and $(SC4)$ from Proposition \ref{SCLittlewood}, this is equal to
\begin{equation}
b^{\BGP(\omega)}q^{|\omega|}\sum_{\underline{\nu}\in \ccp^t} q^{t\displaystyle\lvert\underline{\nu}\rvert}x^{\displaystyle\lvert\underline{\nu}\rvert}\prod_{h\in\mathcal{H}(\underline{\nu})}\rho_1(th)\sum_{h\in\mathcal{H}(\underline{\nu})}\rho_2(th),\label{sumprod}
\end{equation}
where $\lvert\underline{\nu}\rvert:=\displaystyle\sum_{i=0}^{t-1}\lvert\nu^{(i)}\rvert$.

 The product part $q^{t\displaystyle\lvert\underline{\nu}\rvert}x^{\displaystyle\lvert\underline{\nu}\rvert}\prod_{h\in\mathcal{H}(\underline{\nu})}\rho_1(th)$ inside the sum over $\underline{\nu}$ can be rewritten as follows
\begin{multline*}
 \prod_{i=0}^{t/2-1}q^{t\left(\lvert\nu^{(i)}\rvert+\lvert\nu^{(t-1-i)}\rvert\right)}x^{\lvert\nu^{(i)}\rvert+\lvert\nu^{(t-1-i)}\rvert}\prod_{h\in\mathcal{H}(\nu^{(i)})}\rho_1(th)\prod_{h\in\mathcal{H}(\nu^{(t-1-i)})}\rho_1(th).
\end{multline*}

When $t$ is even, as mentioned in the introduction, Proposition~\ref{SCLittlewood}~$(SC2)$ implies that the $t$-quotient $\underline{\nu}$ is uniquely determined by its first $t/2$ components, which are any partitions. It also implies that $\lvert\nu^{(i)}\rvert=\lvert\nu^{(t-1-i)}\rvert$ and $\mathcal{H}(\nu^{(i)})=\mathcal{H}(\nu^{(t-1-i)})$ for any $i\in\lbrace 0,\dots,t/2-1\rbrace$ because sizes and hook-lengths multisets of partitions are invariant by conjugation. Therefore
\begin{align*}
q^{t\displaystyle\lvert\underline{\nu}\rvert}x^{\displaystyle\lvert\underline{\nu}\rvert}\prod_{h\in\mathcal{H}(\underline{\nu})}\rho_1(th)= \prod_{i=0}^{t/2-1}q^{2t\lvert\nu^{(i)}\rvert}x^{2\lvert\nu^{(i)}\rvert}\prod_{h\in\mathcal{H}(\nu^{(i)})}\rho_1^2(th).
\end{align*}

Moreover by application of Proposition~\ref{SCLittlewood}~$(SC2)$ and~$(SC4)$, the sum part $\sum_{h\in\mathcal{H}(\underline{\nu})}\rho_2(th)$ in~\eqref{sumprod} is 
\begin{align*}
\sum_{i=0}^{t/2-1}\left(\sum_{h\in\mathcal{H}(\nu^{(i)})}\rho_2(th)+\sum_{h\in\mathcal{H}(\nu^{(t-1-i)})}\rho_2(th)\right)=2\sum_{i=0}^{t/2-1}\sum_{h\in\mathcal{H}(\nu^{(i)})}\rho_2(th).
\end{align*}

Therefore \eqref{sumprod}, and thus \eqref{term}, become
\begin{multline*}
2 b^{\BGP(\omega)}q^{\vert \omega \vert}\sum_{i=0}^{t/2-1}\left(\sum_{\nu^{(i)}\in\ccp}\displaystyle q^{2t\lvert\nu^{(i)}\rvert}x^{2\lvert\nu^{(i)}\rvert}\prod_{h\in\mathcal{H}(\nu^{(i)})}\rho_1^2(th)\sum_{h\in\mathcal{H}(\nu^{(i)})}\rho_2(th)\right)\\
\times\left(\displaystyle\sum_{\nu\in\ccp} q^{2t\lvert\nu\rvert}x^{2\lvert\nu\rvert}\prod_{h\in\mathcal{H}(\nu)}\rho_1^2(th)\right)^{t/2-1}.
\end{multline*}
Hence we get:
\begin{multline*}
\sum_{\substack{\lambda\in \ccsc\\core_t(\lambda)=\omega}}q^{|\lambda|}x^{|\mathcal{H}_{t}(\lambda)|}b^{\BGP(\lambda)}\prod_{h\in\mathcal{H}_{t}(\lambda)}\rho_1(h)\sum_{h\in\mathcal{H}_{t}(\lambda)}\rho_2(h)\\
=tb^{\BGP(\omega)}q^{|\omega|}\left(f_t\left(x^2q^{2t}\right)\right)^{t/2-1}g_t(x^2q^{2t}).
\end{multline*}

To finish the proof, it remains to show that
\begin{equation}
\sum_{\omega\in\ccsc_{(t)}}q^{\vert \omega\vert}b^{\BGP(\omega)}=\left(q^{2t};q^{2t}\right)_\infty^{t/2}\left(-bq;q^{4}\right)_\infty\left(-q^3/b;q^{4}\right)_\infty.\label{sct}
\end{equation}
For an integer $k$, let $c_{t/2}(k)$ be the number of $t/2$-core partitions of $k$.
Following~\cite{CHS}, define for a nonnegative integer $m$:
 $$\ccsc^{(m)}(n):=\left\lbrace\lambda\in \ccsc(n) : \vert D_1(\lambda)\vert - \vert D_3(\lambda)\vert = (-1)^{m+1}\lceil m/2\rceil\right\rbrace.$$ Setting $p=1$ in \cite[proposition $4.7$]{CHS}, we get that for any integer $m\geq 0$, the number of self-conjugate $t$-core partitions $\omega$ such that $\vert D_1(\omega)\vert - \vert D_3(\omega)\vert = (-1)^{m+1}\lceil m/2\rceil$ is
$$sc_{(t)}^{(m)}(n)=\begin{cases}
c_{t/2}(k) \quad \text{if} \quad n=4k+\frac{m(m+1)}{2},\\
0 \quad \text{otherwise}.
\end{cases}$$
To prove this, the authors define a bijection $\phi^{(m)}$ in \cite[Corollary 4.6]{CHS} between $\omega\in\ccsc^{(m)}_{(t)}$ and $\kappa\in\ccp_{\left(t/2\right)}$ with $|\omega|=4|\kappa|+m(m+1)/2$ and $\kappa$ independent of $m$.

Recall from Lemma \ref{BGSC} that $\BGP(\lambda)=r-s=\vert D_1(\lambda)\vert - \vert D_3(\lambda)\vert$. Therefore 
$$m=\begin{cases}
2 \BGP(\lambda)-1 \quad \text{if}\quad \BGP(\lambda)>0,\\
-2 \BGP(\lambda) \quad \text{if} \quad \BGP(\lambda)\leq 0.
\end{cases}
$$
Hence the bijection $\phi^{(m)}$ maps a $t$-core self-conjugate partition $\omega$ with BG-rank $j$ to a $t/2$-core partition independent of $j$. Then property $(SC5)$ from Proposition~\ref{sc4} implies that $|\omega|=j(2j-1)+4|\kappa|$ with $\kappa$ independent of $j$. Therefore we deduce
\begin{equation}
\sum_{\omega\in\ccsc_{(t)}} q^{|\omega|}b^{\BGP(\omega)}=\sum_{j=-\infty}^{\infty}b^jq^{j(2j-1)}\times\sum_{\kappa\in\ccp_{(t/2)}} q^{4|\kappa|}.\label{coreens}
\end{equation}

Now we compute the sum over $j$. Recall that the Jacobi triple product \cite{HS} can be stated as
\begin{equation*}
\sum_{j=-\infty}^{+\infty}(-1)^jz^jq^{j(j-1)/2}=\left(z;q\right)_{\infty}\left(q/z;q\right)_{\infty}\left(q;q\right)_{\infty}.
\end{equation*}
Therefore, setting $z=-bq$ and then replacing $q$ by $q^4$ in the above identity, yields
\begin{equation}
\sum_{j=-\infty}^{+\infty}b^jq^{j(2j-1)}
=\left(-bq;q^4\right)_{\infty}\left(-q^3/b;q^4\right)_{\infty}\left(q^4;q^4\right)_{\infty}.\label{jacob}
\end{equation}

Finally, to complete the proof of Theorem \ref{bg4multiSC}, it remains to compute the generating function of $t/2$-core partitions which is well-known (see \cite{GKS,Ha}). However we shortly recall its computation. By direct application of the Littlewood decomposition, using $(SC3)$ and the generating series \eqref{gspartitions} for $\ccp$ where $q$ is replaced by $q^{t/2}$, we have for $\omega\in \ccp_{(t/2)}$:
$$
\sum_{\substack{\lambda\in \ccp \\core_{t/2}(\lambda)=\omega}}q^{|\lambda|}=q^{|\omega|}\displaystyle\sum_{i=0}^{t/2-1}\sum_{\nu^{(i)}\in \mathcal{P}} q^{t\lvert\nu^{(i)}\rvert/2}= \frac{q^{|\omega|}}{\left(q^{t/2};q^{t/2}\right)_{\infty}^{t/2}}.
$$
As by \eqref{gspartitions}
$$\frac{1}{\left(q;q\right)_{\infty}}=\sum_{\lambda\in \ccp}q^{|\lambda|}=\sum_{\omega\in \ccp_{(t/2)}}\sum_{\substack{\lambda\in \ccp\\core_{t/2}(\lambda)=\omega}}q^{|\lambda|},
$$
we derive 
\begin{equation}\label{gsp}
\sum_{\omega\in \ccp_{(t/2)}}q^{|\omega|}=\frac{\left(q^{t/2};q^{t/2}\right)_{\infty}^{t/2}}{\left(q;q\right)_{\infty}}.
\end{equation}

Replacing $q$ by $q^4$ in \eqref{gsp}, and using \eqref{coreens} and \eqref{jacob}, this proves \eqref{sct} and the theorem.

\subsection{Special cases}

Here we list useful special cases of Theorem~\ref{bg4multiSC}.

First, by setting $\rho_2=1$, we have the following result.
\begin{corollary}\label{multiSC}
Set $\rho_1$ a function defined on $\mathbb{N}$, and let $t$ be a positive even integer and $f_t$ be defined as in Theorem \ref{bg4multiSC}. Then we have
\begin{multline*}
\sum_{\lambda\in \ccsc}q^{|\lambda|}x^{|\mathcal{H}_{t}(\lambda)|}b^{\BGP(\lambda)}\prod_{h\in\mathcal{H}_{t}(\lambda)}\rho_1(h)\\=\left(f_{t}(x^2q^{2t})\right)^{t/2}\left(q^{2t};q^{2t}\right)_{\infty}^{t/2}\left(-bq;q^4\right)_{\infty}\left(-q^3/b;q^4\right)_{\infty}.
\end{multline*}
\end{corollary}

\begin{proof}
Take $\rho_2=1$ in Theorem \ref{bg4multiSC}. This yields $g_t=\displaystyle\sum_{\nu\in\mathcal{P}}\lvert\nu\rvert q^{|\nu|}\prod_{h\in\mathcal{H}(\nu)}\rho_1(th)^2$. Therefore we get $$g_t(x^2q^{2t})=\frac{x}{2}\frac{d}{dx}f_t(x^2q^{2t}).$$
The right-hand side of Theorem~\ref{bg4multiSC} is then
$$
\frac{t}{2}\left(f_{t}(x^2q^{2t})\right)^{t/2-1}x\frac{d}{dx}f_t(x^2q^{2t})\times
\left(q^{2t};q^{2t}\right)_{\infty}^{t/2}\left(-bq;q^4\right)_{\infty}\left(-q^3/b;q^4\right)_{\infty},
$$
while its left-hand side becomes
$$
\sum_{\lambda\in \ccsc}q^{|\lambda|}\lvert\cch_t(\lambda)\rvert x^{|\mathcal{H}_{t}(\lambda)|} b^{\BGP(\lambda)}\prod_{h\in\mathcal{H}_{t}(\lambda)}\rho_1(h).
$$
We complete the proof by dividing both sides by $x$ and integration with respect to $x$.
\end{proof}

Similarly, as when we take $\rho_1=1$ in Theorem~\ref{bg4multiSC}, then $f_t$ becomes the generating function~\eqref{gspartitions} of $\ccp$ (with $q$ replaced by $x^2q^{2t}$), we immediately derive the following special case.

\begin{corollary}\label{addiSC}
Set $\rho_2$ a function defined on $\mathbb{N}$ and let $t$ be a positive even integer and $g_t$ be defined as in Theorem \ref{bg4multiSC}.
Then
\begin{multline*}
\sum_{\lambda\in \ccsc}q^{|\lambda|}x^{|\mathcal{H}_{t}(\lambda)|}b^{\BGP(\lambda)}\sum_{h\in\mathcal{H}_{t}(\lambda)}\rho_2(h)=tg_t(x^2q^{2t})\\
\times\frac{\left(q^{2t};q^{2t}\right)_{\infty}^{t/2}}{\left(x^2q^{2t};x^2q^{2t}\right)_\infty^{t/2-1}}\left(-bq;q^4\right)_{\infty}\left(-q^3/b;q^4\right)_{\infty}.
\end{multline*}
\end{corollary}

\section{Applications}\label{applications}

In~\cite{HJ}, Han and Ji derive from Theorem~\ref{Multiplication theorem} modular versions of many classical identities for partitions. In this section, we give self-conjugate modular analogues of most of them as consequences of Theorem~\ref{bg4multiSC} and its corollaries. The specificity for $\ccsc$ is that we have to consider $t$ even in all this section. Nevertheless, our results are slightly more general than in~\cite{HJ}, as the $\BGP$-rank can be included in our formulas, although this was only possible in~\cite{HJ} for $t=2$.

\subsection{A generating function}\label{41}
Setting $\rho_1(h)=1$ in Corollary \ref{multiSC}, we derive the following trivariate generating function for $\ccsc$: 

$$
\sum_{\lambda\in \ccsc}q^{|\lambda|}x^{|\mathcal{H}_t(\lambda)|}b^{\BGP(\lambda)}= \frac{\left(q^{2t};q^{2t}\right)_{\infty}^{t/2}}{\left(x^2q^{2t};x^2q^{2t}\right)_\infty^{t/2}}\left(-bq;q^4\right)_{\infty}\left(-q^3/b;q^4\right)_{\infty}.
$$

If we take $x=1$, we obtain the generating function with respect to the $\BGP$-rank for $\ccsc$:
$$
\sum_{\lambda\in \ccsc}q^{|\lambda|}b^{\BGP(\lambda)}= \left(-bq;q^4\right)_{\infty}\left(-q^3/b;q^4\right)_{\infty}.
$$

\subsection{Two classical hook-length formulas} \label{42}
Recall the following hook-length formulas:
\begin{align}
\sum_{\lambda\in\mathcal{P}}q^{|\lambda|}\prod_{h\in\cch}\frac{1}{h^2}&=\exp(q)\label{hl formula},\\
\sum_{\lambda\in\mathcal{P}}q^{|\lambda|}\prod_{h\in\cch}\frac{1}{h}&=\exp\left(q+\frac{q^2}{2}\right).\label{involution}
\end{align}
These formulas are direct consequences of the Robinson--Schensted--Knuth correspondence (see for example \cite{Sta99} p.324). Again, we can use Corollary~\ref{multiSC} to find self-conjugate modular versions for them. The difference between the case of $\ccp$ treated in \cite{HJ} and the case of self-conjugate partitions is that now $\rho_1$ is replaced by its square leading to applications with $1/h$ and $1/\sqrt{h}$ instead of $1/h^2$ and $1/h$.
 
The modular $\ccsc$ version of~\eqref{hl formula} is as follows.
\begin{corollary}\label{hl ana}
 For $t$ an even positive integer, we have:
\begin{multline*}
\sum_{\lambda\in \ccsc}q^{|\lambda|}x^{|\mathcal{H}_t(\lambda)|}b^{\BGP(\lambda)}\prod_{h\in\mathcal{H}_t(\lambda)}\frac{1}{h}\\=\left(q^{2t};q^{2t}\right)_{\infty}^{t/2}\left(-bq;q^4\right)_{\infty}\left(-q^3/b;q^4\right)_{\infty}\exp\left(\frac{x^2q^{2t}}{2t}\right).
\end{multline*}
\end{corollary}
\begin{proof}
Taking $\rho_1(h)=1/h$ in Corollary~\ref{multiSC}, we have by using~\eqref{hl formula}:
$$f_t(q)=\exp\left(\frac{q}{t^2}\right).$$
\end{proof}

Setting $x=1$ and comparing coefficients $b^0$ on both sides of Corollary \ref{hl ana}, we get:
$$
\displaystyle\sum_{\substack{\lambda\in \ccsc\\\BGP(\lambda)=0}}q^{|\lambda|}\prod_{h\in\mathcal{H}_t(\lambda)}\frac{1}{h}=\frac{\left(q^{2t};q^{2t}\right)_{\infty}^{t/2}}{\left(q^4;q^4\right)_{\infty}}\exp\left(\frac{q^{2t}}{2t}\right).$$

Note that in \cite{HJ}, a similar formula was given for $\ccp$ only when $t=2$.

By identification of the coefficients of $b^jx^{2n}q^{2tn+j(2j-1)}$ on both sides of Corollary~\ref{hl ana}, we have for all integers $j$ and all nonnegative integers $n$:
$$
\sum_{\substack{\lambda\in\ccsc,\lambda \vdash 2tn+j(2j-1)\\ \substack{\lvert\cch_t(\lambda)\rvert=2n}\\\BGP(\lambda)=j}}\prod_{h\in\cch_t(\lambda)}\frac{1}{h}=\frac{1}{n!2^nt^n}.
$$
The case $j=0$ is the same result as P\'etr\'eolle (\cite{MP} Corollary 4.24):
$$
\sum_{\substack{\lambda\in\ccsc,\lambda\vdash 2tn\\ \lvert\cch_t(\lambda)\rvert=2n}}\prod_{h\in\cch_t(\lambda)}\frac{1}{h}=\frac{1}{n!2^nt^n},
$$
as the conditions on $\lambda$ in the summation necessarily imply by the Littlewood decomposition that $\omega=\emptyset$, which is equivalent to $\BGP(\lambda)=0$.

Now we prove the following modular $\ccsc$ version of \eqref{involution}.
\begin{corollary}\label{involsc}
For $t$ an even positive integer, we have:
\begin{multline*}
\sum_{\lambda\in \ccsc}q^{|\lambda|}x^{|\mathcal{H}_t(\lambda)|}b^{\BGP(\lambda)}\prod_{h\in\mathcal{H}_t(\lambda)}\frac{1}{h^{1/2}}\\=
\left(q^{2t};q^{2t}\right)_{\infty}^{t/2}\left(-bq;q^4\right)_{\infty}\left(-q^3/b;q^4\right)_{\infty}\exp\left(\frac{x^2q^{2t}}{2}+\frac{x^4q^{4t}}{4t}\right).
\end{multline*}
\end{corollary}
\begin{proof}
Take $\rho_1(h)=1/h^{1/2}$ in Corollary~\ref{multiSC}. Then by direct application of~\eqref{involution}, we have:
$$f_t(q)=\exp\left(\frac{q}{t}+\frac{q^2}{2t^2}\right).
$$
\end{proof}

Setting $x=1$ and comparing coefficients $b^0$ on both sides of Corollary~\ref{involsc}, we derive:
$$
\sum_{\substack{\lambda\in \ccsc\\\BGP(\lambda)=0}}q^{|\lambda|}\prod_{h\in\mathcal{H}_t(\lambda)}\frac{1}{h^{1/2}}=
\frac{\left(q^{2t};q^{2t}\right)_{\infty}^{t/2}}{\left(q^4;q^4\right)_{\infty}}\exp\left(\frac{q^{2t}}{2}+\frac{q^{4t}}{4t}\right).
$$

On the other hand, by comparing coefficients of $q^{2tn+j(2j-1)}x^{2n}b^j$ on both sides of Corollary~\ref{involsc}, we have the following result, which is true for all integers $j$ and all positive integers $n$:
$$
\sum_{\substack{\lambda\in\ccsc\\ \lambda\vdash 2tn+j(2j-1)\\ \substack{\vert\cch_t(\lambda)\vert=2n\\\BGP(\lambda)=j}}}\prod_{h\in\mathcal{H}_t(\lambda)}\frac{1}{h^{1/2}}=\frac{1}{2^n}\sum_{k=0}^{\lfloor n/2\rfloor}\frac{1}{k!(n-2k)!t^k}.$$

\subsection{The Han--Carde--Loubert--Potechin--Sanborn formula}\label{43}
The following formula is an interpolation between \eqref{hl formula} and \eqref{involution} conjectured by Han in \cite{Hancon} and proved by Carde--Loubert--Potechin--Sanborn in \cite{CLPS} :
\begin{equation}
\sum_{\lambda\in\mathcal{P}}q^{|\lambda|}\displaystyle\prod_{h\in\cch}\frac{1}{h}\frac{1+z^h}{1-z^h}=\exp\left(\frac{1+z}{1-z}q+\frac{q^2}{2}\right).\label{invol twist}
\end{equation}

Here is a modular $\ccsc$ version of \eqref{invol twist}.
\begin{corollary}\label{invol ana}
For $t$ an even positive integer, for any complex number $z$, we have:
\begin{multline*}
\sum_{\lambda\in \ccsc}q^{|\lambda|}x^{|\mathcal{H}_t(\lambda)|}b^{\BGP(\lambda)}\prod_{h\in\mathcal{H}_t(\lambda)}\left(\frac{1}{h}\frac{1+z^h}{1-z^h}\right)^{1/2}\\=\left(q^{2t};q^{2t}\right)_{\infty}^{t/2}\left(-bq;q^4\right)_{\infty}\left(-q^3/b;q^4\right)_{\infty}\exp\left(\frac{1+z^t}{1-z^t}\frac{x^2q^{2t}}{2}+\frac{x^4q^{4t}}{4t}\right).
\end{multline*}
\end{corollary}
\begin{proof}
Take $\rho_1(h)=\left(\frac{1}{h}\frac{1+z^h}{1-z^h}\right)^{1/2}$ in Corollary \ref{multiSC}. By direct application of~\eqref{invol twist}, we have:
$$
f_t(q)=\exp\left(\frac{1+z^t}{1-z^t}\frac{q}{t}+\frac{q^2}{2t^2}\right).
$$
\end{proof}

\subsection{The Nekrasov--Okounkov formula}\label{44}
In \cite{Pe}, P\'etr\'eolle discovered and proved analogues of the Nekrasov--Okounkov formula \eqref{NOdebut} for $\ccsc$ and $\ccdd$ (which is a slight deformation of $\ccsc$).  
In his work, a sign appears combinatorially, which corresponds to the algebraic sign in the associated Littlewood formulas for Schur functions~\cite[11.9.5 p.238]{Little}. Here it is possible to avoid the sign and only use~\eqref{NOdebut} with Theorem~\ref{bg4multiSC} to derive a modular $\ccsc$ version of Nekrasov--Okounkov type when $t$ is even. This is given in Corollary~\ref{NOsign} that we prove below. In Section~\ref{refine} we will prove refined versions of our results which take the signs into account.

\begin{proof}[Proof of Corollary~\ref{NOsign}]
Take $\rho_1(h)=\left(1-z/h^2\right)^{1/2}$ in Corollary~\ref{multiSC}, we have by application of~\eqref{NOdebut}:
$$f_t(q)=\left(q;q\right)_\infty^{z/t^2-1}.
$$
The conclusion follows when this result is plugged in the right-hand side of Corollary~\ref{multiSC}.
\end{proof}
By setting $z=-c^2/x^2$ and letting $x\rightarrow 0$, the left-hand side of Corollary \ref{NOsign} becomes:
$$
\sum_{\lambda\in\ccsc}q^{|\lambda|}b^{\BGP(\lambda)}\prod_{h\in\cch_t(\lambda)}\frac{c}{h}.$$
On the right hand side, the three first terms remain unchanged, while we can write for all $j\geq 1$:
$$
\left(1-x^{2j}q^{2tj}\right)^{(z/t-t)/2}= \exp\left(\frac{1}{2}\left(\frac{c^2}{tx^2}+t\right)\displaystyle\sum_{k\geq 1}\frac{x^{2jk}q^{2tjk}}{k}\right),$$
therefore
\begin{align*}
(x^2q^{2t};x^2q^{2t})_\infty^{(z/t-t)/2}&=\exp\left(\frac{1}{2}\left(\frac{c^2}{tx^2}+t\right)\displaystyle\sum_{k\geq 1}\frac{x^{2k}q^{2tk}}{k(1-x^{2k}q^{2tk})}\right)\\
&=\exp\left(\frac{c^2 q^{2t}}{2t}+O(x^2)\right)\\
&_{\overrightarrow{x\to 0}}\exp\left(\frac{c^2 q^{2t}}{2t}\right).
\end{align*}

Therefore
$$
\sum_{\lambda\in\ccsc}q^{|\lambda|}b^{\BGP(\lambda)}\prod_{h\in\cch_t(\lambda)}\frac{c}{h}=\left(q^{2t};q^{2t}\right)_{\infty}^{t/2}\left(-bq;q^4\right)_{\infty}\left(-q^3/b;q^4\right)_{\infty}\exp\left(\frac{c^2 q^{2t}}{2t}\right),
$$
which is equivalent to the identity in Corollary~\ref{hl ana}.

\subsection{The Bessenrodt--Bacher--Manivel formula}\label{45}

The following formula deals with power sums of hook-lengths. Its proof is based on a result due to Bessenrodt, Bacher and Manivel~\cite{Bess,BachMani} which provides a mapping, for any couple of positive integers $j<k$, the total number of occurrences of the part $k$ among all partitions of $n$ to the number of boxes whose hook-type is $(j,k-j-1)$. In~\cite{HJ}, Han and Ji explain that this result can be embedded in the following generalization, which is true for any complex number $\beta$:
\begin{equation}\label{beta}
\sum_{\lambda\in\mathcal{P}}q^{|\lambda|}\sum_{h\in\cch}h^{\beta}=\frac{1}{(q;q)_\infty} \sum_{k\geq 1}k^{\beta+1}\frac{q^k}{1-q^k}.
\end{equation}

The modular $\ccsc$ version of~\eqref{beta} takes the following form.
\begin{corollary}\label{betalemma}
For any complex number $\beta$ and $t$ an even positive integer, we have:
\begin{multline*}
\sum_{\lambda\in \ccsc}q^{|\lambda|}x^{|\mathcal{H}_t(\lambda)|}b^{\BGP(\lambda)}\sum_{h\in\mathcal{H}_t(\lambda)}h^{\beta}\\=\frac{\left(q^{2t};q^{2t}\right)_{\infty}^{t/2}}{\left(x^2q^{2t};x^2q^{2t}\right)_{\infty}^{t/2}}\left(-bq;q^4\right)_{\infty}\left(-q^3/b;q^4\right)_{\infty}\sum_{k\geq 1}\left(tk\right)^{\beta+1}\frac{x^{2k}q^{2kt}}{1-x^{2k}q^{2kt}}.
\end{multline*}
\begin{proof}
Take $\rho_2(h)=h^{\beta}$ in Corollary~\ref{addiSC} and then use~\eqref{beta} to compute:
$$
g_t(q)=\frac{t^{\beta}}{\left(q;q\right)_\infty}\sum_{k\geq 1}k^{\beta+1}\frac{q^k}{1-q^k}.
$$
\end{proof}
\end{corollary}

\subsection{The Okada--Panova formula}\label{46}

The following formula is the generating function form of the Okada--Panova formula, which was conjectured by Okada and proved by Panova in~\cite{Pa}:
\begin{equation}\label{OP}
\sum_{\lambda\in\mathcal{P}}q^{|\lambda|}\prod_{h\in\cch}\frac{1}{h^2}\sum_{h\in\cch(\lambda)}\prod_{i=1}^r \left(h^2-i^2\right)=C(r)q^{r+1}\exp(q),
\end{equation}
where 
$$C(r):=\frac{1}{2(r+1)^2}\bi{2r}{r}\bi{2r+2}{r+1}.$$
To find a modular $\ccsc$ version of~\eqref{OP}, we want to use Theorem~\ref{bg4multiSC} with $\rho_1(h)=1/h$ and $\rho_2(h)=\displaystyle\prod_{i=1}^r\left(h^2-i^2\right)$. Recall from~\cite{HJ} that:
\begin{equation}\label{fOP}
 f_{\alpha}(q):=\sum_{\lambda\in \ccp}q^{|\lambda|}\prod_{h\in\cch(\lambda)}\frac{1}{\left(\alpha h\right)^2}=\exp\left(\frac{q}{\alpha^2}\right).
\end{equation}
We also define as in~\cite{HJ}:
\begin{equation*}
g_{\alpha}(q):=\sum_{\lambda\in \ccp}q^{|\lambda|}\prod_{h\in\cch}\frac{1}{\left(\alpha h\right)^2}\sum_{h\in\cch(\lambda)}\prod_{i=1}^r\left(\left(\alpha h\right)^2-i^2\right).
\end{equation*}

In order to evaluate $g_{\alpha}(q)$, Han and Ji introduce the polynomials defined by the following relations:
\begin{align*}
B_{r,0}(\alpha)&=\prod_{j=1}^r \left(\alpha^2-j^2\right),\\
B_{r,k}(\alpha)&=\left[\alpha^2\left(k+1\right)^2-r^2\right]B_{r-1,k}(\alpha)+\alpha^2 B_{r-1,k-1}(\alpha) \quad\text{for} \quad k\in\lbrace 1,\dots,r-1\rbrace,\\
B_{r,r}(\alpha)&=\alpha^{2r}.
\end{align*}

This enables them to rewrite $g_{\alpha}(q)$ in~\cite[Proposition 8.2]{HJ} as:
\begin{equation}\label{gOP}
g_{\alpha}(q)=\exp\left(\frac{q}{\alpha^2}\right)\sum_{k=0}^rB_{r,k}(\alpha)C(k)\left(\frac{q}{\alpha^2}\right)^{k+1}.
\end{equation}

We prove the following modular $\ccsc$ version of~\eqref{OP}.
\begin{corollary}\label{cor:OP}
For any positive integer $r$ and $t$ an even positive integer, we have:
\begin{multline*}
\sum_{\lambda\in\ccsc}q^{|\lambda|}x^{|\mathcal{H}_t(\lambda)|}b^{\BGP(\lambda)}\prod_{h\in\cch_t(\lambda)}\frac{1}{h}\sum_{h\in\cch_t(\lambda)}\prod_{i=1}^r \left(h^2-i^2\right)\\
=t\left(q^{2t};q^{2t}\right)_{\infty}^{t/2}\left(-bq;q^4\right)_{\infty}\left(-q^3/b;q^4\right)_{\infty}\\ \times\exp \left(\frac{x^2q^{2t}}{2t}\right)\sum_{k=\lceil (r-t+1)/t\rceil}^r B_{r,k}(t)C(k)\left(\frac{x^2q^{2t}}{t^2}\right)^{k+1}.
\end{multline*}
\end{corollary}

\begin{proof}
Take $\rho_1(h)=1/h$ and $\rho_2(h)=\displaystyle\prod_{i=1}^r\left(h^2-i^2\right)$ in Theorem~\ref{bg4multiSC} and $\alpha=t$ in~\eqref{fOP} and~\eqref{gOP} to rewrite $f_t$ and $g_t$, respectively.
\end{proof}

\subsection{The Stanley--Panova formula}\label{47}
Panova and Stanley proved the following formula~\cite{Pa,Sta09}:
\begin{equation}\label{lasteq}
n!\sum_{\lambda\vdash n} \prod_{h\in\cch(\lambda)}\frac{1}{h^2}\sum_{h\in \cch(\lambda)}h^{2k}=\sum_{i=0}^kT(k+1,i+1)C(i)\prod_{j=0}^i(n-j)
\end{equation}
where $T(k,i)$ is a central factorial number \cite[ex.5.8]{Sta99} defined for $k\geq 1$ and $i\geq 1$ by:
\begin{align*}
T(k,0)&=T(0,i)=0,\quad T(1,1)=1,\\
T(k,i)&=i^2T(k-1,i)+T(k-1,i-1)\quad
\text{for}\quad (k,i)\neq (1,1).
\end{align*}

By setting $\rho_1(h)=1/(\alpha h)$ and $\rho_2(h)=(\alpha h)^{2k}$, we have as in~\eqref{fOP}
\begin{equation}\label{fSP}
 f_{\alpha}(q)=\sum_{\lambda\in \ccp}q^{\vert\lambda\vert}\prod_{h\in\cch(\lambda)}\frac{1}{\left(\alpha h\right)^2}=\exp\left(\frac{q}{\alpha^2}\right),
\end{equation}
 and by using~\eqref{lasteq}
\begin{multline}\label{gSP}
g_{\alpha}(q)=\sum_{\lambda\in \ccp}q^{\vert\lambda\vert}\prod_{h\in\cch}\frac{1}{\left(\alpha h\right)^2}\sum_{h\in\cch(\lambda)}\alpha^{2k} h^{2k}\\
=\alpha^{2k}\exp\left(\frac{q}{\alpha^2}\right)\sum_{i=0}^k T(k+1,i+1)C(i)\left(\frac{q}{\alpha^2}\right)^{i+1}.
\end{multline}

Now we prove the following modular $\ccsc$ version of \eqref{lasteq}.
\begin{corollary}\label{cor:SP}
For any positive integer $k$ and $t$ an even positive integer, we have:
\begin{multline*}
\sum_{\lambda\in\ccsc}q^{|\lambda|}x^{|\mathcal{H}_t(\lambda)|}b^{\BGP(\lambda)}\prod_{h\in\cch_t(\lambda)}\frac{1}{h}\sum_{h\in\cch_t(\lambda)}h^{2k}\\
=t^{2k+1}\left(q^{2t};q^{2t}\right)_{\infty}^{t/2}\left(-bq;q^4\right)_{\infty}\left(-q^3/b;q^4\right)_{\infty}\\\exp \left(\frac{x^2q^{2t}}{2t}\right)\sum_{i=0}^k T(k+1,i+1)C(i)\left(\frac{x^2q^{2t}}{t^2}\right)^{i+1}.
\end{multline*}
\end{corollary}

\begin{proof}
Take $\rho_1(h)=1/h$ and $\rho_2(h)=h^{2k}$ in Theorem \ref{bg4multiSC} and $\alpha=t$ in \eqref{fSP} and \eqref{gSP} to rewrite $f_t$ and $g_t$, respectively.
\end{proof}

\section{Signed refinements}\label{refine}

In~~\cite{MP}, P\'etr\'eolle proved the following $\ccsc$ Nekrasov--Okounkov type formula similar to~\eqref{NOdebut}, which stands for any complex number $z$: 
\begin{equation}\label{formulemathias}
\sum_{\lambda\in\ccsc}\delta_\lambda q^{\lvert\lambda\rvert}\prod_{\substack{u\in\lambda\\h_u\in\cch(\lambda)}}\left(1-\frac{2z}{h_u\varepsilon_{u}}\right)=\left(\frac{\left(q^2;q^2\right)_\infty^{z+1}}{\left(q;q\right)_\infty}\right)^{2z-1}.
\end{equation}

Here, $\delta_\lambda$ and $\epsilon_{u}$ are signs depending on the partition $\lambda$, and the position of any box $u$ in its Ferrers diagram (written $u\in\lambda$ in the above formula), respectively. If the Durfee square of $\lambda$ has size $d$, then one simply defines \textbf{$\delta_{\lambda}$}$:=(-1)^d$. Recall that this sign also has an algebraic meaning regarding Littlewood summations for Schur functions indexed by partitions in $\ccsc$. 
Next, for any partition $\lambda\in\ccsc$ and a box $u=(i,j)\in\lambda$, \textit{$\varepsilon_{u}$} is defined as $-1$ if $u$ is a box strictly below the diagonal of the Ferrers diagram and as $1$ otherwise. 

Our goal in this section is to prove a multiplication-addition theorem similar to Theorem~\ref{bg4multiSC} including the above signs. Nevertheless one can notice that for $\lambda\in \ccsc$, we have actually $\delta_\lambda=(-1)^{|\lambda|}$. Indeed, by Lemma~\ref{BGSC} in Section~\ref{lit} and by definition of the $\BGP$-rank, one has $|\lambda|\equiv r-s\pmod 2$; and moreover $d=r+s$ by definition of $D_1(\lambda)$ and $D_3(\lambda)$. This means that the sign $\delta_\lambda$ can readily be omitted, by replacing $q$ by $-q$ in formulas like~\eqref{formulemathias} and their modular analogues.

Recall that Lemma \ref{lemdurf} allows to determine the position with respect to the main diagonal of the Ferrers diagram, thanks to the correspondence between a box of $\lambda$ and a pair of indices of the corresponding word $s(\lambda)$. Next, to include the sign $\varepsilon$, we will need a refinement of Proposition~\ref{Littlewood} $(P3)$, which is an immediate consequence of the Littlewood decomposition: for $\lambda\in\ccp$ and any box $u\in\lambda$ with hook-length $h_u\in\cch_t(\lambda)$ (here $t$ is any positive integer), there exists a unique $k\in\lbrace 0,\dots,t-1\rbrace$ and a unique box $u_k\in\nu^{(k)}$ such that $h_u=th_{u_k}$, where $h_{u_k}$ is the hook-length of $u_k$ in the partition $\nu^{(k)}$. We will say that the box $u_k$ is {\em associated to} the box $u$. We have the following result for self-conjugate partitions.
\begin{lemma}\label{box}
Set $\lambda \in\ccsc$, let $t$ be a positive even integer. Set $u \in \lambda$ such that $h_u\in \cch_t(\lambda)$. Then the following properties hold true:
\begin{enumerate}
\item The box $u$ does not belong to the main diagonal of $\lambda$.
\item The application $u\mapsto u'$, where $u'$ is the symmetric of $u$ with respect to the main diagonal of $\lambda$, is well-defined on $\lambda$, bijective and satisfies \linebreak $h_{u'}=h_u\in \cch_t(\lambda)$ and $\varepsilon_{u}=-\varepsilon_{u'}$.
\item If $u_k$ and $u_l$ are the boxes associated to $u$ and $u'$ respectively, then \linebreak $l=t-1-k$.
\end{enumerate}
\end{lemma}

\begin{proof}
For any $\ccsc$ partition, all hook-lengths of boxes on the main diagonal are odd numbers. As $t$ is even, the result $(1)$ follows.

Next $(2)$ is a direct consequence of $(1)$ and the definitions of $\ccsc$ and $\varepsilon_u$. 

Finally, to prove $(3)$ we need to explicit the bijection between the coordinates of a box of $\lambda$ and a pair of indices of the corresponding word $s(\lambda)=\left(c_i \right)_{i\in\bbbz}$ defined in Section~\ref{lit}. Let us introduce the two following sets:
\begin{align*}
I&:=\lbrace i\in\bbbz\,|\, c_i=1\;\text{and}\; \exists j\in \bbbz\;\text{such that}\; i<j,\,c_j=0\rbrace,\\
J&:=\lbrace j\in\bbbz\,|\, c_j=0 \;\text{and}\; \exists i\in \bbbz\;\text{such that}\; i<j,\,c_i=1\rbrace.
\end{align*}

By definition of $s(\lambda)$, those sets are finite. Therefore one can write $I=\lbrace i_1,\dots,i_{\lambda_1}\rbrace$ and $J=\lbrace j_1,\dots,j_{\lambda'_1}\rbrace$ such that the sequence $(i_a)_{a\in\lbrace 1,\dots,\lambda_1'\rbrace}$ (resp. $(j_b)_{b\in\lbrace 1,\dots,\lambda_1\rbrace}$) is strictly increasing (resp. strictly decreasing).

Let $F(\lambda)$ be the Ferrers diagram of $\lambda$ and define the application 
$$\begin{array}{l|rcl}
\Psi: & F(\lambda) & \to & I \times J \\
& (x,y) & \mapsto & (i_y,j_x).
\end{array}$$
Note that $\Psi$ is injective by monotony of the sequences $(i_a)$ and $(j_b)$.

Recall that $\lambda\in\ccsc$ translates in terms of the associated word to:
\begin{equation}\label{eqword}
c_j=1-c_{-1-j}\;\;\;\forall j\in\bbbn.
\end{equation}

This implies that $\lvert I\rvert =\lambda'_1=\lvert J\rvert =\lambda_1$. Let $\psi:I\to\psi(I)$ be the application such that $\psi(i_m):=-1-i_m$. The aforementioned property actually guarantees that $\psi(I)\subset J$. As $\lvert I\rvert=\lvert J\rvert$, we deduce that $\psi$ is bijective. Moreover, as $(i_a)_{a\in\lbrace 1,\dots,\lambda'_1\rbrace}$ is strictly increasing, we derive that $\left(\psi(i_a)\right)$ is strictly decreasing and for any \linebreak $a\in\lbrace  1,\dots,\lambda_1'=\lambda_1\rbrace$, we have $j_a=-1-i_a$.

Suppose that $(i_y,j_x)\in \Psi(F(\lambda))$ is such that $i_y\equiv k \pmod t$ and $j_x\equiv k \pmod t$.
By~\eqref{mot} and the bijectivity of $\psi$ sending $(i_a)$ to $(j_b)$, we have that $(i_x,j_x)\in\Psi(F(\lambda))$ and $i_x \equiv t-1-k \pmod t$ and $j_y \equiv t-1-k \pmod t$.
As $u'$ has coordinates $(i_x,j_y)$ and is associated to the box $u_l$, we derive that $l=t-1-k$, which concludes the proof.
\end{proof}

\subsection{A signed addition-multiplication theorem}

We will now prove a generalization of Theorem \ref{bg4multiSC} which includes the sign mentioned above.

\begin{theorem}\label{signmultisc}
Set $t$ an even integer and let $\tilde{\rho_1},\tilde{\rho_2}$ be two functions defined on $\bbbz\times \lbrace-1,1\rbrace$. Set also $f_t(q),g_t(q)$ the formal power series defined by:
\begin{align*}
f_t(q)&:=\sum_{\nu\in\mathcal{P}}q^{|\nu|}\prod_{h\in\cch(\nu)}\tilde{\rho_1}(th,1)\tilde{\rho_1}(th,-1),\\
g_t(q)&:=\sum_{\nu\in\mathcal{P}}q^{|\nu|}\prod_{h\in\cch(\nu)}\tilde{\rho}_1(th,1)\tilde{\rho}_1(th,-1)\sum_{h\in\cch(\nu)}\left(\tilde{\rho_2}(th,1)+\tilde{\rho_2}(th,-1)\right).
\end{align*}
Then we have
\begin{multline*}
\sum_{\lambda\in \ccsc}q^{|\lambda|}x^{|\mathcal{H}_t(\lambda)|}b^{\BGP(\lambda)}\prod_{\substack{u\in\lambda\\h_u\in\cch_t(\lambda)}}\tilde{\rho_1}(h_u,\varepsilon_u)\sum_{\substack{u\in\lambda\\h_u\in\cch_t(\lambda)}}\tilde{\rho_2}(h_u,\varepsilon_u)\\
=\frac{t}{2}\left(f_t(x^2q^{2t})\right)^{t/2-1}g_t(x^2q^{2t})\left(q^{2t};q^{2t}\right)_{\infty}^{t/2}\left(-bq;q^4\right)_{\infty}\left(-q^3/b;q^4\right)_{\infty}.
\end{multline*}
\end{theorem}

\begin{proof}
The proof follows the same steps as the one of Theorem~\ref{bg4multiSC}, but now~\eqref{sumprod} becomes
\begin{equation}\label{pointdedepart}
b^{\BGP(\omega)}q^{|\omega|}\sum_{\underline{\nu}\in \ccp^t} q^{t\displaystyle\lvert\underline{\nu}\rvert}x^{\displaystyle\lvert\underline{\nu}\rvert}\prod_{u\in\underline{\nu}}\tilde{\rho}_1(th_u,\varepsilon_u)\sum_{u\in\underline{\nu}}\tilde{\rho}_2(th_u,\varepsilon_u),
\end{equation}
where $\omega$ is in $\ccsc_{(t)}$.
The product part $q^{t\displaystyle\lvert\underline{\nu}\rvert}x^{\displaystyle\lvert\underline{\nu}\rvert}\prod_{u\in\underline{\nu}}\tilde{\rho}_1(th_u,\varepsilon_u)$ inside the sum over $\underline{\nu}$ can be rewritten as follows
\begin{multline*}\label{refpourthmsign}
 \prod_{i=0}^{t/2-1}q^{t\left(\lvert\nu^{(i)}\rvert+\lvert\nu^{(t-1-i)}\rvert\right)}x^{\lvert\nu^{(i)}\rvert+\lvert\nu^{(t-1-i)}\rvert}\prod_{h\in\mathcal{H}(\nu^{(i)})}\tilde{\rho}_1(th,1)\tilde{\rho}_1(th,-1).
\end{multline*}
Indeed, by Lemma~\ref{box}, each box $u\in\nu^{(i)}$, with $0\leq i\leq t-1$, is bijectively paired with a box $u'\in\nu^{(t-1-i)}$ satisfying $\tilde{\rho}_1(th_{u'},\varepsilon_{u'})=\tilde{\rho}_1(th_{u},-\varepsilon_{u})$. The sum part $\sum_{u\in\underline{\nu}}\tilde{\rho}_2(th_u,\varepsilon_u)$ in~\eqref{pointdedepart} can be rewritten in a similar fashion. We leave the rest of the proof to the reader as the remaining computations are similar to the ones used to prove Theorem~\ref{bg4multiSC}.
\end{proof}

Note that Theorem~\ref{bg4multiSC} is a consequence of Theorem~\ref{signmultisc}, by choosing \linebreak $\tilde{\rho}_1(a,\varepsilon)=\rho_1(a)$ and $\tilde{\rho}_2(a,\varepsilon)=\rho_2(a)$. Moreover by choosing $\tilde{\rho}_1=1$ or $\tilde{\rho}_2=1$, we have special cases similar to Corollaries~\ref{multiSC} and~\ref{addiSC}. However we will only highlight the case where $\tilde{\rho}_2=1$, as this one yields interesting consequences.

\begin{corollary}\label{multisignSC}
Set $\tilde{\rho_1}$ a function defined on $\mathbb{Z}\times\lbrace -1,1\rbrace$, and let $t$ be a positive even integer and $f_t$ be defined as in Theorem~\ref{signmultisc}. Then we have
\begin{multline*}
\sum_{\lambda\in \ccsc}q^{|\lambda|}x^{|\mathcal{H}_{t}(\lambda)|}b^{\BGP(\lambda)}\prod_{\substack{u\in\lambda\\h_u\in\cch_t(\lambda)}}\tilde{\rho}_1(h_u,\varepsilon_{u})\\=\left(f_{t}(x^2q^{2t})\right)^{t/2}\left(q^{2t};q^{2t}\right)_{\infty}^{t/2}\left(-bq;q^4\right)_{\infty}\left(-q^3/b;q^4\right)_{\infty}.
\end{multline*}
\end{corollary}

\subsection{Applications}

We have applications similar to the ones obtained in Sections~\ref{41}--\ref{47}. However we only highlight the cases concerning Sections~\ref{41}--\ref{44}, which are the most interesting in our opinion and are all derived from Corollary~\ref{multisignSC}.

First note that the generating series obtained with $\tilde{\rho}_1=1$ is the same as the one in Section~\ref{41}. 

Next, when $t$ is an even positive integer and $\lambda\in\ccsc$, notice that 
$$\prod_{\substack{u\in\lambda\\h_u\in\cch_t(\lambda)}}\varepsilon_u=(-1)^{\cch_t(\lambda)/2}.$$
Therefore the specialization $\tilde{\rho}_1(a,\varepsilon)=1/(a\varepsilon)$ yields a hook-length formula \linebreak equivalent to the one in Corollary~\ref{hl ana} when $x$ is replaced by $x\sqrt{-1}$. Similarly, the specialization $\tilde{\rho}_1(a,\epsilon)=1/(a\epsilon)^{1/2}$ yields a hook-length formula equivalent to the one in Corollary~\ref{involsc} when $x$ is replaced by $x\sqrt[4]{-1}$.

Now the signed modular analogue of~\eqref{invol twist} is as follows.
\begin{corollary}\label{invol sign ana}
For $t$ an even positive integer, for any complex number $z$, we have:
\begin{multline*}
\sum_{\lambda\in \ccsc}q^{|\lambda|}x^{|\mathcal{H}_t(\lambda)|}b^{\BGP(\lambda)}\prod_{\substack{u\in\lambda\\h_u\in\cch_t(\lambda)}}\frac{1}{h_u^{1/2}}\frac{1+\varepsilon_{u} z^{h_u}\sqrt{-1}}{1-\varepsilon_{u} z^{h_u}}\\=\left(q^{2t};q^{2t}\right)_{\infty}^{t/2}\left(-bq;q^4\right)_{\infty}\left(-q^3/b;q^4\right)_{\infty}\exp\left(\frac{1+z^t}{1-z^t} \frac{x^2q^{2t}}{2}+\frac{x^4q^{4t}}{4t}\right).
\end{multline*}
\end{corollary}

\begin{proof}
Take 
$$\tilde{\rho}_1(a,\varepsilon)=\frac{1}{a^{1/2}}\frac{1+\varepsilon z^{a}\sqrt{-1}}{1-\varepsilon z^{a}}$$
 in Corollary~\ref{multisignSC} and use the identity $\tilde{\rho}_1(a,1)\tilde{\rho}_1(a,-1)=(1+z^a)/(a(1-z^a))$ and Formula~\eqref{invol twist} to conclude.
\end{proof}

The signed modular SC analogue of the Nekrasov--Okounkov formula~\eqref{NOdebut}, which is actually a modular analogue of~\eqref{formulemathias}, is the following.
\begin{corollary}\label{NOsignfinal}
For any complex number $z$ and $t$ an even positive integer, we have:
\begin{multline*}
\sum_{\lambda\in\ccsc}q^{|\lambda|}x^{|\mathcal{H}_t(\lambda)|}b^{\BGP(\lambda)}\prod_{\substack{u\in\lambda\\h_u\in\cch_t(\lambda)}}\left(1-\frac{z}{h_u\varepsilon_{u}}\right)\\=\left(q^{2t};q^{2t}\right)_{\infty}^{t/2}\left(-bq;q^4\right)_{\infty}\left(-q^3/b;q^4\right)_{\infty}\left(x^{2}q^{2t};x^{2}q^{2t}\right)_\infty^{(z^2/t-t)/2}.
\end{multline*}
\end{corollary}
\begin{proof}
Take $\tilde{\rho}_1(a,\varepsilon)=1-z/(a\varepsilon)$ in Corollary~\ref{multisignSC}, then use the identity \linebreak $\tilde{\rho}_1(a,1)\tilde{\rho}_1(a,-1)=1-z^2/a^2$ and~\eqref{NOdebut} to conclude.
\end{proof}
Note that taking $b=1$ in the above formula, one gets P\'etr\'eolle's result~\cite[Th\'eor\`eme~4.22]{MP}, in which $q,y,z$ have to be replaced by $-q,x,z/t$, respectively.

By identifying coefficients on both sides of the previous formula, we get the following consequence.
\begin{corollary}\label{NOodd} For all positive integers $n$ and all integers $j$, we have
\begin{equation}
\sum_{\substack{\lambda\in\ccsc,\lambda\vdash 2nt+j(2j-1)\\ \BGP(\lambda)=j}}\prod_{h\in\cch_t(\lambda)}\frac{1}{h}\sum_{h\in\cch_t(\lambda)}\frac{h^2}{2}=\frac{1}{2^{n}t^{n-1}(n-1)!}(t+3n-3). \label{nohook}
\end{equation}
\end{corollary}
\begin{proof}
By Lemma \ref{box}, the left-hand side of Corollary \ref{NOsignfinal} can be rewritten as follows
\begin{equation}\label{interstep}
\sum_{\lambda\in\ccsc}q^{|\lambda|}x^{|\mathcal{H}_t(\lambda)|}b^{\BGP(\lambda)}\prod_{\substack{\substack{u\in\lambda\\h_u\in\cch_t(\lambda)}\\\varepsilon_u=1}}\left(1-\frac{z^2}{h_u^2}\right).
\end{equation}

The left-hand side of \eqref{nohook} is the coefficient of $q^{2tn+j(2j-1)}x^{2n}b^j(-z^2)^{n-1}$ in \eqref{interstep}. Using the following identity
\begin{equation*}
\prod_{m\geq 1}\frac{1}{1-q^m}=\exp\left(\sum_{k\geq 1}\frac{q^k}{k(1-q^k)}\right),
\end{equation*}
the right-hand side of Corollary \ref{NOsignfinal} can be rewritten:
\begin{multline*}
R=\frac{\left(q^{2t};q^{2t}\right)_{\infty}^{t/2}}{\left(x^{2}q^{2t};x^{2}q^{2t}\right)_\infty^{t/2}}\left(-bq;q^4\right)_{\infty}\left(-q^3/b;q^4\right)_{\infty}\exp\left(\frac{-z^2}{2t}\sum_{k\geq 1}\frac{(x^2q^{2t})^k}{k(1-(x^2q^{2t})^k)}\right).
\end{multline*}
Thus, by also using \eqref{jacob}, our desired coefficient is equal to
\begin{align*}
&\left[q^{2tn+j(2j-1)}x^{2n}b^j(-z^2)^{n-1}\right]R\\
&=\left[q^{2tn}x^{2n}(-z^2)^{n-1}\right]\frac{\left(q^{2t};q^{2t}\right)_{\infty}^{t/2}}{\left(x^{2}q^{2t};x^{2}q^{2t}\right)_\infty^{t/2}\left(q^{4};q^{4}\right)_{\infty}}\exp\left(\frac{-z^2}{2t}\sum_{k\geq 1}\frac{(x^2q^{2t})^k}{k(1-(x^2q^{2t})^k)}\right)\\
&=\left[q^{2tn}x^{2n}\right]\frac{1}{2^{n-1}t^{n-1}(n-1)!}\frac{1}{\left(x^{2}q^{2t};x^{2}q^{2t}\right)_\infty^{t/2}}\left(\sum_{k\geq 1}\frac{(x^2q^{2t})^k}{k(1-(x^2q^{2t})^k)}\right)^{n-1}\\
&=\left[q^{2t}x^{2}\right]\frac{1}{2^{n-1}t^{n-1}(n-1)!}\left(1+\frac{t}{2}x^2 q^{2t}\right)\left(\frac{1}{1-x^2 q^{2t}}+\frac{x^2 q^{2t}}{2\left(1-(x^2 q^{2t})^2\right)}\right)^{n-1}\\
&=\frac{1}{2^{n-1}t^{n-1}(n-1)!}\left(\frac{t}{2}+\frac{3(n-1)}{2}\right)\\
&= \frac{1}{2^{n}t^{n-1}(n-1)!}(t+3n-3).
\end{align*}

\end{proof}

Corollary \ref{NOodd} could also be derived from Corollary \ref{cor:SP} by setting $k=1$ and comparing the coefficients of $q^{2tn+j(2j-1)}x^{2n}b^j$ on both sides.

\section{The odd case}\label{sec:odd}

In this section, we analyse the case where $t$ is a positive odd integer. 
Recall that in this case the Littlewood decomposition can be written as follows:
\begin{eqnarray}
\lambda\in\ccsc &\mapsto &\left(\omega,\underline{\tilde{\nu}},\mu\right)\in\ccsc_{(t)}\times\ccp^{(t-1)/2}\times\ccsc.\label{toddlast}
\end{eqnarray}

When $t$ is odd, Formula~(3.4) in~\cite{GKS} gives a connection between the $\BGP$-rank of a partition, and its $t$-quotient and its $t$-core. However the formula implies a dependence between $t$-core and $t$-quotient, which is not convenient for multiplication-addition type theorems. This is why we will formulate multiplication-addition type theorems without the $\BGP$-rank.

 Moreover, because of the partition $\mu\in\ccsc$ appearing in~\eqref{toddlast}, more difficulties arise which make a general result less elegant than in the even case. Even if it is possible to prove a general odd analogue to Theorem~\ref{bg4multiSC} (without the $\BGP$-rank), formulas on self-conjugate partitions would be required to derive interesting applications. These are, to our knowledge, missing in the literature. This is why we will focus here on a subset of self-conjugate partitions for which $\mu$ is empty, which, as will be explained, is algebraically interesting.

For a fixed positive odd integer $t$, let us define 
$$\BGP^t:=\lbrace \lambda\in\ccsc, \Phi_t(\lambda)=\left(\omega,\underline{\nu}\right)\in\ccsc_{(t)}\times\ccp^{t}\;\mbox{with}\;\nu^{((t-1)/2)}=\emptyset\rbrace.$$
Note that $\lambda$ is in $\BGP^t$ if and only if the partition $\mu$ is empty in~\eqref{toddlast}.
 Following~\cite{Bern}, we also define for an odd prime number $p$, the set of self-conjugate partitions with no diagonal hook-length divisible by $p$: 
 $$\BGP_p:=\lbrace\lambda\in\ccsc \mid \forall i \in\lbrace 1,\dots,d\rbrace, p\nmid h_{(i,i)}\rbrace.$$
    Algebraically, this set yields interesting properties in representation theory of the symmetric group over a field of characteristic $p$, see for instance~\cite{BruG, Bern}. Combinatorially, it is natural to extend this definition to a set $\BGP_t$ for any positive odd number $t$.

The following result explains the connection between the two above sets and is proved in~\cite[Lemma $3.4$]{BruG} for any prime number $p$. Nevertheless, we give a proof here to generalize it to any positive odd integer $t$.
\begin{lemma}\label{geneBruG}
For any positive odd integer $t$, we have:
$$\BGP^t=\BGP_t.$$
\end{lemma}

\begin{proof}
Take $\lambda\in\ccsc\setminus \BGP_t$. There exists $(x,x)\in\lambda$ such that $t\mid h_{(x,x)}$. Recall that $h_{(x,x)}$ is necessarily odd. Take $m$ such that $h_{(x,x)}= t(2m+1)$. Let $(i_x,j_x)\in\bbbz^2$ be the pair of indices in $s(\lambda)$ associated with the box $(x,x)$. Then $j_x\geq 0$ and $i_x<0$. Moreover, by \eqref{eqword}, one has $i_x=-j_x-1$. As $h_{(x,x)}=j_x-i_x$, we get $h_{(x,x)}=2j_x+1$. This yields $2j_x+1=t(2m+1)$. Therefore we have
$$j_x=tm+\frac{t-1}{2}.$$

This implies that there exists a sequence ``$10$" in the subword \linebreak $(c_{kt+(t-1)/2})_{k\in\bbbz}=s(\mu)$, where $\mu=\nu^{((t-1)/2)}$ is the partition uniquely defined by the Littlewood decomposition. Hence $\mu\neq \emptyset$ and therefore $\lambda\notin\BGP^t$.

Conversely, let $\lambda\in\ccsc\setminus \BGP^t$. So $\mu\neq \emptyset$.
  Set $s(\lambda)=(c_k)_{k\in\mathbb{Z}}$ the corresponding word. Remark that $\mu\neq\emptyset$ is equivalent to the existence of $i_1\in\bbbn$ such that $c_{ti_1+(t-1)/2}=0$ and $c_{-ti_1+(t-1)/2}=0$. This implies that there exists a hook of length $t(2i_1+1)$ which is on the main diagonal of $\lambda$. Therefore $\lambda\notin\BGP_t$.
\end{proof}

We now prove the following result which is the analogue of Theorem~\ref{signmultisc} for $t$ odd, restricted to the set $\BGP^t=\BGP_t$.
\begin{theorem}\label{thm:bgt}
Let $t$ be a positive odd integer and set $\tilde{\rho_1},\tilde{\rho_2}$ two functions defined on $\mathbb{Z}\times\lbrace -1,1\rbrace$. Let $f_t$ and $g_t$ be the formal power series defined in Theorem \ref{signmultisc}. Then we have

\begin{multline*}
\displaystyle\sum_{\lambda\in\BGP^t}q^{|\lambda|}x^{|\mathcal{H}_t(\lambda)|}\displaystyle\prod_{\substack{u\in\lambda\\h_u\in\mathcal{H}_t(\lambda)}}\tilde{\rho}_1(h_u,\varepsilon_{u})\sum_{\substack{u\in\lambda\\h_u\in\mathcal{H}_t(\lambda)}}\tilde{\rho}_2(h_u,\varepsilon_{u})\\= (t-1)\left(f_t(x^2q^{2t})\right)^{(t-3)/2}g_t(x^2q^{2t})\frac{\left(q^{2t};q^{2t}\right)^{(t-1)/2}_{\infty}\left(-q;q^2\right)_{\infty}}{\left(-q^t;q^{2t}\right)_{\infty}}.
\end{multline*}
\end{theorem}

\begin{proof}
The proof follows the same lines as the ones of Theorems~\ref{bg4multiSC} and~\ref{signmultisc} but with $b=1$. Here $t$ is odd and the summation on the left-hand side is over partitions in $\BGP^t$ (therefore $\nu^{((t-1)/2)}=\mu=\emptyset$), so the Littlewood decomposition shows that, in our situation,~\eqref{pointdedepart} takes the form 
\begin{equation*}
q^{|\omega|}\sum_{\underline{\nu}\in \ccp^{t-1}} q^{t\displaystyle\lvert\underline{\nu}\rvert}x^{\displaystyle\lvert\underline{\nu}\rvert}\prod_{u\in\underline{\nu}}\tilde{\rho}_1(th_u,\varepsilon_u)\sum_{u\in\underline{\nu}}\tilde{\rho}_2(th_u,\varepsilon_u),
\end{equation*}
where $\omega$ is a fixed $t$-core in $\BGP^t$.
Next we can transform the above expression by using Proposition~\ref{SCLittlewood} and Lemma~\ref{box}: although the latter was proved in the $t$ even case only, it is possible to extend it to $t$ odd for partitions $\lambda\in\BGP^t$, by noticing that a box $u$ is on the main diagonal of $\lambda$ and satisfies $h_u\in\mathcal{H}_t(\lambda)$ only if $u$ is associated by the Littlewood decomposition to a box in $\nu^{((t-1)/2)}=\mu$, which is empty in our situation. Therefore we can proceed as in the proof of Theorem~\ref{signmultisc}, but the factor $t$ in Theorem~\ref{signmultisc} now becomes $t-1$. 

The remaining part to finish the proof is the computation of the generating series of partitions $\omega$ in $\BGP^t$ that are $t$-cores, that are partitions in the set $\BGP^t_{(t)}$. As remarked in~\cite{AO}, the generating series of $\BGP^t$ is given by
\begin{equation}
\sum_{\lambda\in\BGP^t}q^{\vert \lambda\vert}=\prod_{\substack{k\geq 1\\t\nmid 2k+1}}(1+q^{2k+1})=\frac{\left(-q;q^2\right)_\infty}{\left(-q^t;q^{2t}\right)_\infty}.\label{bernsg}
\end{equation}
By using Proposition~\ref{SCLittlewood}~$SC3$ of the Littlewood decomposition and the generating series~\eqref{gspartitions} for partitions, the left-hand side of~\eqref{bernsg} can be rewritten as
$$
\sum_{\omega\in\BGP^t_{(t)}}q^{\vert \omega\vert}\left(\sum_{\nu\in\ccp}q^{2t\vert\nu\vert}\right)^{(t-1)/2}=\frac{1}{\left(q^{2t};q^{2t}\right)_\infty^{(t-1)/2}}\sum_{\omega\in\BGP^t_{(t)}}q^{\vert \omega\vert}.$$
Hence the generating series of $\BGP^t_{(t)}$ is 
$$\sum_{\omega\in\BGP^t_{(t)}}q^{\vert \omega\vert}=\frac{\left(q^{2t};q^{2t}\right)^{(t-1)/2}_{\infty}\left(-q;q^2\right)_{\infty}}{\left(-q^t;q^{2t}\right)_{\infty}}.$$
The rest of the proof follows the exact same steps as for Theorem~\ref{bg4multiSC}, without taking the $\BGP$-rank into account.
\end{proof}

Note that by taking $\tilde{\rho}_1(a,\varepsilon)=\rho_1(a)$ and $\tilde{\rho}_2(a,\varepsilon)=\rho_2(a)$ in the above result, we get an analogue of Theorem~\ref{bg4multiSC} for $t$ odd and $b=1$, restricted to the set $\BGP^t=\BGP_t$. 

We now derive applications of Theorem~\ref{thm:bgt} in the same spirit as the ones proved in Sections~\ref{applications} and~\ref{refine}, but for odd $t$. As the specializations are the same here, we do not give details for the proofs.

 First, our bivariate generating function takes the form:
$$
\sum_{\lambda\in \BGP^t}q^{|\lambda|}x^{|\mathcal{H}_t(\lambda)|}= \frac{\left(q^{2t};q^{2t}\right)^{(t-1)/2}_{\infty}\left(-q;q^2\right)_{\infty}}{\left(x^2q^{2t};x^2q^{2t}\right)_\infty^{(t-1)/2}\left(-q^t;q^{2t}\right)_{\infty}}.
%
$$

Next, the odd analogues of Corollaries~\ref{hl ana} and~\ref{involsc} for $\BGP^t$ are summarized in the following result.
\begin{corollary}\label{hl odd-invol odd}
 For $t$ a positive odd integer, we have:
\begin{multline*}
\sum_{\lambda\in \BGP^t}q^{|\lambda|}x^{|\mathcal{H}_t(\lambda)|}\prod_{h\in\mathcal{H}_t(\lambda)}\frac{1}{h}=
\frac{\left(q^{2t};q^{2t}\right)^{(t-1)/2}_{\infty}\left(-q;q^2\right)_{\infty}}{\left(-q^t;q^{2t}\right)_{\infty}}
\exp\left((t-1)\frac{x^2q^{2t}}{2t^2}\right),
\end{multline*}
and 
\begin{multline*}
\sum_{\lambda\in \BGP^t}q^{|\lambda|}x^{|\mathcal{H}_t(\lambda)|}\prod_{h\in\mathcal{H}_t(\lambda)}\frac{1}{h^{1/2}}\\=
\frac{\left(q^{2t};q^{2t}\right)^{(t-1)/2}_{\infty}\left(-q;q^2\right)_{\infty}}{\left(-q^t;q^{2t}\right)_{\infty}}
\exp\left((t-1)\left(\frac{x^2q^{2t}}{2t}+\frac{x^4q^{4t}}{4t^2}\right)\right).
\end{multline*}
\end{corollary}
The odd version of Corollary~\ref{invol sign ana} is as follows.
\begin{corollary}\label{invol ana odd}
For $t$ a positive odd integer, for any complex number $z$, we have:
\begin{multline*}
\sum_{\lambda\in \BGP^t}q^{|\lambda|}x^{|\mathcal{H}_t(\lambda)|}\prod_{\substack{u\in\lambda\\h_u\in\cch_t(\lambda)}}\frac{1}{h_u^{1/2}}\frac{1+\varepsilon_{u}\sqrt{-1} z^{h_u}}{1-\varepsilon_{u} z^{h_u}}\\=
\frac{\left(q^{2t};q^{2t}\right)^{(t-1)/2}_{\infty}\left(-q;q^2\right)_{\infty}}{\left(-q^t;q^{2t}\right)_{\infty}}
\exp\left((t-1)\left(\frac{1+z^t}{1-z^t}  \frac{x^2q^{2t}}{2t}+\frac{x^4q^{4t}}{4t^2}\right)\right).
\end{multline*}
\end{corollary}
Now the odd version of the modular signed Nekrasov--Okounkov type formula given in Corollary~\ref{NOsignfinal} is given bellow.
\begin{corollary}\label{NOsignodd}
For $t$ a positive odd integer, for any complex number $z$, we have:
\begin{multline*}
\sum_{\lambda\in\BGP^t}q^{|\lambda|}x^{|\mathcal{H}_t(\lambda)|}\prod_{\substack{u\in\lambda\\h_u\in\cch_t(\lambda)}}\left(1-\frac{z}{h_u\varepsilon_{u}}\right)\\=\frac{\left(q^{2t};q^{2t}\right)^{(t-1)/2}_{\infty}\left(-q;q^2\right)_{\infty}}{\left(-q^t;q^{2t}\right)_{\infty}}
\left(x^{2}q^{2t};x^{2}q^{2t}\right)_\infty^{(t-1)(z^2/t^2-1)/2}.
\end{multline*}
\end{corollary}
Finally, the odd analogues of Corollaries~\ref{betalemma},~\ref{cor:OP} and~\ref{cor:SP} are given in the three results below.
\begin{corollary}\label{betalemma odd}
For any complex number $\beta$ and $t$ a positive odd integer, we have:
\begin{multline*}
\sum_{\lambda\in \BGP^t}q^{|\lambda|}x^{|\mathcal{H}_t(\lambda)|}\sum_{h\in\mathcal{H}_t(\lambda)}h^{\beta}=(t-1)\frac{\left(q^{2t};q^{2t}\right)^{(t-1)/2}_{\infty}\left(-q;q^2\right)_{\infty}}{\left(x^2q^{2t};x^2q^{2t}\right)_{\infty}^{(t-1)/2}\left(-q^t;q^{2t}\right)_{\infty}}\\
\times\sum_{k\geq 1}\left(tk\right)^{\beta+1}k\frac{x^{2k}q^{2kt}}{1-x^{2k}q^{2kt}}.
\end{multline*}
\end{corollary}

\begin{corollary} For any positive integer $r$ and $t$ a positive odd integer, we have:
\begin{multline*}
\sum_{\lambda\in\BGP^t}q^{|\lambda|}x^{|\mathcal{H}_t(\lambda)|}\prod_{h\in\cch_t(\lambda)}\frac{1}{h}\sum_{h\in\cch_t(\lambda)}\prod_{i=1}^r \left(h^2-i^2\right)=(t-1)\frac{\left(q^{2t};q^{2t}\right)^{(t-1)/2}_{\infty}\left(-q;q^2\right)_{\infty}}{\left(-q^t;q^{2t}\right)_{\infty}}\\ 
\times\exp\left( (t-1)\left(\frac{x^2q^{2t}}{2t^2}\right)\right)\sum_{k=\lceil (r-t+1)/t \rceil}^r B_{r,k}(t)C(k)\left(\frac{x^2q^{2t}}{t^2}\right)^{k+1}.
\end{multline*}
\end{corollary}

\begin{corollary}
For any positive integer $k$ and $t$ a positive odd integer, we have:
\begin{multline*}
\sum_{\lambda\in\BGP^t}q^{|\lambda|}x^{|\mathcal{H}_t(\lambda)|}\prod_{h\in\cch_t(\lambda)}\frac{1}{h}\sum_{h\in\cch_t(\lambda)}h^{2k}=(t-1)t^{2k}\frac{\left(q^{2t};q^{2t}\right)^{(t-1)/2}_{\infty}\left(-q;q^2\right)_{\infty}}{\left(-q^t;q^{2t}\right)_{\infty}}\\ 
\times\exp \left((t-1)\frac{x^2q^{2t}}{2t^2}\right)\sum_{i=0}^k T(k+1,i+1)C(i)\left(\frac{x^2q^{2t}}{t^2}\right)^{i+1}.
\end{multline*}
\end{corollary}



\end{document}